\documentclass[11pt,english]{article}
\usepackage[a4paper,bindingoffset=0.2in,%
            left=1in,right=1in,top=1in,bottom=1in,%
            footskip=.25in]{geometry}
\usepackage{xspace}
\usepackage{tablists}
\title{An important paper}
\usepackage{lmodern}

\usepackage{fancybox}
\cornersize {2}
\usepackage{geometry}
\usepackage[T1]{fontenc}
\usepackage{cite}
\usepackage{fancyhdr}
\usepackage{stmaryrd}
\usepackage{bbm}
\usepackage{textcomp}
\usepackage[dvips]{epsfig}
\usepackage{amscd}
\usepackage{amssymb}
\usepackage{amsthm}
\usepackage{latexsym}
\usepackage{mathrsfs}
\usepackage{relsize}
\usepackage{upgreek}
\usepackage{amsmath}
\usepackage{amsfonts}
\usepackage{amssymb}
\usepackage{dsfont}
\usepackage{graphicx}
\usepackage{caption}
\usepackage{subcaption}
\usepackage{color}
\usepackage{minitoc}
\usepackage{calligra}
\usepackage[toc,page,title]{appendix}
\usepackage[utf8]{inputenc}
\usepackage[]{lastpage}
\usepackage{tikz}
\usepackage{tikz-cd}
\usetikzlibrary{arrows,calc,shapes,decorations.pathreplacing}
\usetikzlibrary{decorations.pathmorphing}
\usetikzlibrary{shapes,arrows}
\usetikzlibrary{patterns}

\usetikzlibrary{arrows,shapes,positioning}
\usetikzlibrary{decorations.markings}
\tikzstyle arrowstyle=[scale=2]
\tikzstyle directed=[postaction={decorate,decoration={markings,
    mark=at position .65 with {\arrow[arrowstyle]{stealth}}}}]
\tikzstyle reverse directed=[postaction={decorate,decoration={markings,
    mark=at position .65 with {\arrowreversed[arrowstyle]{stealth};}}}]
\usepackage{enumerate}
\usepackage{tikz}
\theoremstyle{plain}
\newtheorem{thm}{Theorem}[section]
\theoremstyle{plain}
\newtheorem{lem}[thm]{Lemma}
\newtheorem{prop}[thm]{Proposition}
\newtheorem{cor}[thm]{Corollary}

\theoremstyle{definition}
\newtheorem{defi}[thm]{Definition}

\newtheorem{RQ}[thm]{Remark}
\usepackage{float}
\usepackage{textcomp}
\RequirePackage[bookmarks,%
                colorlinks,
                urlcolor=blue,%
                citecolor=blue,%
                linkcolor=black,%
                hyperfigures,%
                pdfcreator=LaTeX,%
                breaklinks=true,%
         pdfpagelayout=SinglePage,%
                bookmarksopen=true,%
                bookmarksopenlevel=2]{hyperref}
\newcommand{\abs}[1]{\ensuremath{\left|#1\right|}}

\usepackage[utf8]{inputenc}
\usepackage[english]{babel}
\usepackage[affil-it]{authblk}
\DeclareUnicodeCharacter{0301}{\'{e}}

\newcommand\smallO{
  \mathchoice
    {{\scriptstyle\mathcal{O}}}
    {{\scriptstyle\mathcal{O}}}
    {{\scriptscriptstyle\mathcal{O}}}
    {\scalebox{.7}{$\scriptscriptstyle\mathcal{O}$}}
  }

\usepackage{tocbasic}
\DeclareTOCStyleEntry[
  beforeskip=.2em plus 1pt,
  pagenumberformat=\textbf
]{tocline}{section}

\begin{document}
\setcounter{tocdepth}{1}
\title{\textbf{Birkhoff normal form in low regularity for the nonlinear quantum harmonic oscillator }}
\author{\scshape{Charbella} Abou Khalil\footnote{Nantes Universit\'e, CNRS, Laboratoire de Math\'ematiques Jean Leray, LMJL, F-44000 Nantes, France 
Email: \href{mailto: Charbella.AbouKhalil@univ-nantes.fr}{Charbella.AbouKhalil@univ-nantes.fr}
 }}


\date{}
\maketitle
\begin{abstract}
Given small initial solutions of the nonlinear quantum harmonic oscillator on $\mathbb{R}$, we are interested in their long time behavior in the energy space which is an adapted Sobolev space. We perturbate the linear part by $V$ taken as multiplicative potentials, in a way that the linear frequencies satisfy a non-resonance condition. More precisely, we prove that for almost all potentials $V,$ the low modes of the solution are almost preserved for very long times. \vspace{0.5cm}

\end{abstract}

\tableofcontents

\section{Introduction}
Over the past half century, the theory of partial differential equations has been an active research field and has mainly focused on studying the local or global existence of solutions within well-chosen functional spaces. Nevertheless, as the theory has advanced, researchers began exploring additional questions, particularly regarding the qualitative behavior of solutions once their existence is established. More precisely, given a small initial datum and a non-resonant\footnote{The eigenvalues ${(\Lambda_j)}_{j\geq 1}$ of the linearized vector field enjoy a Diophantine condition which implies rational independency. For the precise definition refer to Definition \fcolorbox{red}{ white}{\ref{dep1}}.} dispersive Hamiltonian partial differential equation on a bounded domain, \[ i\partial_t u = \partial_{\bar{u}}H(u),\] with $H$ denoting a smooth Hamiltonian having 0 as an elliptic equilibrium, we seek to understand the properties of the solution in a Sobolev space $H^s$ over long periods of time. \vspace{0.3cm}\\
In this article, we succeed in removing the smoothness assumption from \cite{zbMATH05655526} and in proving the almost global preservation of the low modes of small solutions of the quantum harmonic oscillator with a perturbation taken as multiplicative potentials. This is the first time such a result is obtained for this equation (nonlinear Schr\"{o}dinger equation on $\mathbb{R}$ with confinement), considered with multiplicative potentials and in the low regularity framework, presenting a new significant perspective. The author in \cite{zbMATH06187752} considers a similar equation and uses Chelkak--Kargaev--Korotyaev’s results about the inverse spectral problem of harmonic oscillator to prove some non-resonance condition on the spectrum. Consequently, the considered solution could be at most $\widehat{H}^1$ (refer to (\fcolorbox{red}{ white}{\ref{Sobolev}})). Thus, he was not able to apply the standard methods, and no stability result was deduced. However, here, we work with non trivial spectral estimates to establish a strong non-resonance condition (see Theorem \fcolorbox{red}{ white}{\ref{ber51}}) allowing us to apply a non-classical version of the Birkhoff normal form procedure (details are found in Theorem \fcolorbox{red}{ white}{\ref{N102}}), which results in stability.

\subsection{{The model}}\label{model} Our main focus is to study the long time behavior of small solutions of the perturbed quantum harmonic oscillator in one dimension in the adapted Sobolev spaces $\widehat{H}^s$ (see (\fcolorbox{red}{ white}{\ref{Sobolev}})) with low regularity ($s$ small). This system is of great importance in quantum physics (refer for instance to \cite{zbMATH06144197}) and is defined by the following Schr\"{o}dinger equation 
\begin{align}\label{N1}\tag{NLS}
 \left\{
    \begin{array}{ll}
  i \partial_t u(t,x)=-\partial_{xx}u(t,x) + x^2u(t,x)+V(x)u(t,x) \pm \abs{u(t,x)}^{2p}u(t,x)\\
    {u}_{|t=0}=u^{(0)}\in \widehat{H}^1(\mathbb{R}),
     \end{array}
\right.
    \end{align}
where $(t,x)\in \mathbb{R}\times\mathbb{R}$, $p\geq 1$ and $V(x)$ is a real-valued potential. Moreover, $\pm$ added to the nonlinearity term refers to the focusing and defocusing cases. For $V=0,$ the linear part of the equation simply describes a quantum harmonic oscillator on $\mathbb{R},$ denoted by $T:= -\partial_{xx} + x^2.$ Notice that (\fcolorbox{red}{ white}{\ref{N1}}) can be seen as a perturbation of the linear equation 
\begin{align}\label{ber5}  i \partial_t u(t,x)&=T u(t,x). \end{align}
 It is well known that the spectrum of this operator is an increasing sequence ${(\lambda_j)}_{j\geq 1}$ given by $\lambda_j = 2j-1$. More precisely, we have  $T h_j = (2j-1) h_j $
with ${(h_j)}_{j\geq 1}$ being the Hermite functions and forming an orthonormal basis of $L^2(\mathbb{R})$ (we refer the reader to Chapter 6 in \cite{zbMATH07312855}).
Moreover, these eigenvalues are completely resonant: since they are integers, they are not rationally independent.
Now, notice that (\fcolorbox{red}{ white}{\ref{ber5}}) can be written as a Hamiltonian system with a quadratic Hamiltonian 
\begin{align*} H_0(u) &= \frac{1}{2} \left({\Vert \partial_x u \Vert}_{L^2}^2 + {\Vert x u \Vert}_{L^2}^2\right) \simeq {\Vert u \Vert}_{\widehat{H}^1}^2. \end{align*}
On the other hand, the frequencies ${(\lambda_j)}_{j\geq 1}$ appear in $H_0$, and thus we also have for $V=0$,
 $$H_0(u)= \frac{1}{2} \sum_{j\geq 1}\lambda_j \abs{u_j}^2 \simeq {\Vert u \Vert}_{h^{1/2}}^2 \quad \text{with} \quad u_j(t) = \int_{\mathbb{R}}u(t,x)h_j(x)\, \mathrm{d}x.$$ 
 
We identify $u\in \widehat{H}^1$ and its Hermite sequence ${(u_j)}_{j\geq 1} \in h^{1/2}$ (see \fcolorbox{red}{ white}{\ref{notations}}).
Our goal will be to adapt a suitable Birkhoff normal form theorem for the nonlinear quantum harmonic oscillator and establish its dynamical consequence, in order to reach the main result presented in the next section. To do so, we require a non-resonance condition on the spectrum of operator $T+V$. To get rid of resonances, we move the eigenvalues ${(\lambda_j)}_{j\geq 1}$ to obtain rational independency. More precisely, we draw smooth potentials $V$ to guarantee that, almost surely, the spectrum is strongly non-resonant (see Theorem \fcolorbox{red}{ white}{\ref{ber51}}) in the following sense\footnote{This definition does not deal with multiplicities (we are in the case of distinct frequencies). }: 
 \begin{defi}\label{dep1}(Strong $N,r$ non-resonance)
Consider frequencies $w\in \mathbb{R}^{\mathbb{N}^*}$, $r\geq 1$ and $N\geq 1$. We say that $w$ are strongly $N,r$ non-resonant if there exists $\beta_{r,N}>0,$ such that for all $1\leq r^* \leq r$, $\sigma \in (\mathbb{Z}^*)^{r^*}$ and all $j\in (\mathbb{N}^*)^{r^*}$ with $j_1 < \cdots < j_{r^*} ,$ $j_1\leq N$ and $\abs{\sigma_1}+\cdots+\abs{\sigma_{r^*}} \leq r$, we have \[ \abs{\sigma_1w_{j_1}+ \cdots + \sigma_{r^*}w_{j_{r^*}}} \geq \beta_{r,N}.\]
\end{defi}
One of the challenges resulting from the new approach we present in our work is verifying this non-resonance condition associated with spectral analysis, where non-trivial estimates are established in  \fcolorbox{red}{ white}{\ref{N80}},  \fcolorbox{red}{ white}{\ref{lemimp8}},  \fcolorbox{red}{ white}{\ref{ikea2}} and  \fcolorbox{red}{ white}{\ref{ikea5}}. They require attention to detail especially when proving that the eigenfunctions and eigenvalues of the operator $T+V$, which we denote by ${(\Lambda_j)}_{j\geq 1}$ and ${(\psi_j)}_{j\geq 1}$ respectively, remain close to those of $T$.
The spectral analysis and properties of the eigenvalues and eigenfunctions will be explained rigorously in Section \fcolorbox{red}{ white}{\ref{z1}}.

     \subsection{{Main results and comments}}\label{resultsection}
  We are interested in the actions for the nonlinear quantum harmonic oscillator describing the dynamics or the amplitudes of the modes of the solution and given for $V\neq0$ as \[I_j(u) = \abs{u_j(t)}^2 \quad \text{ with } \quad  u_j(t) = \int_{\mathbb{R}}u(t,x)\psi_j(x)\, \mathrm{d}x\] 
 where we recall that ${(\psi_j)}_{j\geq 1}$ are the eigenfunctions of the operator $T+V.$ Notice that the actions $I_j$ are preserved by the linear part of the Schr\"{o}dinger equation (refer to (\fcolorbox{red}{ white}{\ref{N1}})). Nevertheless, once we turn on the nonlinear perturbation, we can expect some exchange of energy (see for example \cite{colliander2010transfer}), and the question of preservation of the actions then arises. 
 
 To state the main result, it is crucial to mention that equation (\fcolorbox{red}{ white}{\ref{N1}}) is globally well-posed for small solutions in $\widehat{H}^1$ (see Section \fcolorbox{red}{ white}{\ref{z4}}). In the following theorem, we consider multiplicative potentials, and we specify the dynamics of the solution over very long times in low regularity. 
     
  \begin{thm}\label{ikea1000}
Let $N \geq 1,$ $r\geq  p+1$ arbitrarily large, $\nu >0$ and let $V \in \widehat{H}^1 \cap \mathscr{C}^2$ such that the spectrum of $T+V$ is strongly $N,r$ non-resonant (refer to Definition \fcolorbox{red}{ white}{\ref{dep1}}). Then, there exist $\varepsilon_0>0$ depending on ${\Vert V \Vert}_{\widehat{H}^1}$ and a constant $C>0$ depending on $(N,r,V, \nu)$ such that if we set $\varepsilon := {\Vert u^{(0)} \Vert}_{\widehat{H}^1}\leq \varepsilon_0,$ the global solution of (\fcolorbox{red}{ white}{\ref{N1}}) satisfies 
\begin{align}\label{mainresult}
 \abs{t} < \varepsilon ^{-2r} \quad \text{and} &\quad 1\leq j \leq N \implies \abs{\abs{u_j(t)}^2 - \abs{u_j(0)}^2}\leq C \varepsilon^{2p+2-\nu}
 \end{align}
with $2p+2$ being the order of the Hamiltonian non-linearity and $u_j(t) = \int_{\mathbb{R}}u(t,x)\psi_j(x)\, \mathrm{d}x.$ 
   \end{thm}

We emphasize that in the above result, the set of potentials $V$ is not empty. More precisely, for almost all $V$, the non-resonant assumption is satisfied. To see this, we present another result as the following theorem:
   
\begin{thm}\label{ber51}
Let $V$ be defined randomly on $\mathbb{R}$ as 
\begin{align}\label{ber50}
    V(x)&= \sum_{k\geq 1} g_kh_k(x\sqrt{2}) P_k
    \end{align}
where $g_k \sim \mathcal{N}(0,1)$ are some independent Gaussian variables and $P \in \widehat{H}^3$ is a given weight such that $P_k \in \mathbb{R}_+^*$. Then, for all $r\geq 1$ and $N\geq 1,$ almost surely $V \in \widehat{H}^1 \cap \mathscr{C}^2$ and, provided that ${\Vert V \Vert}_{\widehat{H}^1}\lesssim_r  N^{-1/6},$ the frequencies of the operator T+V are strongly $N,r$ non-resonant in the sense of Definition \fcolorbox{red}{ white}{\ref{dep1}}.
\end{thm}   
As a consequence, Theorem \fcolorbox{red}{ white}{\ref{ikea1000}} applies. Notice that Theorem \fcolorbox{red}{ white}{\ref{ber51}} makes sense because we prove that $\mathbb{P}({\Vert V \Vert}_{\widehat{H}^1} < \lambda )>0 \text{ for all } \lambda>0$ in Lemma \fcolorbox{red}{ white}{\ref{lems}} of the Appendix.\vspace{0.3cm}

\textbf{Comments regarding the main results.} \\  
-- Estimate (\fcolorbox{red}{ white}{\ref{mainresult}}) means that we prove the almost global preservation of the low actions over very long times $\abs{t} < \varepsilon ^{-2r}$ with $r$ arbitrarily large. Note that this is trivial for time scales $\abs{t}\leq \varepsilon^{-2p}$ indicating that the dynamics of (\fcolorbox{red}{white}{\ref{N1}}) remain close to the dynamics of the linearized equation, even in the case of a vanishing potential. However, the conservation of the actions is not trivial on longer scales (for $r$ arbitrarily large), which is the case here.\\
--Fortunately and without the additional smoothness constraint, we obtained a result of the same kind as in \cite{zbMATH05655526}. Here, in order to avoid the resonances, we perturbate the eigenvalues by adding multiplicative potentials instead of Hermite multipliers, resulting in a much more complicated spectral aspect. \\
--The transition to low regularity results in the loss of information concerning the high modes of the solution. The Birkhoff normal form theorem, developed by \cite{bernier2021birkhoff}, concerns only the behavior of the first $N$ modes, similarly to our result. Furthermore, the strong $N,r$ non-resonance condition stated in Theorem \fcolorbox{red}{white}{\ref{ber51}} clearly provides a relation between the potential $V$ and $N$ as well as indicates that the larger the number of modes $N$ we wish to control, the smaller the potential $V$ has to be. Unlike Theorem 1.10 in \cite{bernier2021birkhoff}, the number of modes we control does not depend only on the size of the initial datum.\\
-- In the classical Birkhoff normal form theorem, a standard non-resonant argument is required (for instance, refer to \cite{zbMATH05117851}) in order to avoid\footnote{On the contrary, for references regarding the exchange of energy in NLS see \cite{grebert2012resonant} and \cite{zbMATH05865483}.} the exchange of energy between modes and deduce the stability. However, since we are working with a non-smooth solution, we will use a stronger condition (Theorem \fcolorbox{red}{white}{\ref{ber51}}) allowing us to remove much more terms from the original Hamiltonian. \\
--It seems interesting to mention that the term $\varepsilon^{-\nu}$ in the estimate of Theorem \fcolorbox{red}{white}{\ref{ikea1000}} is due to truncation and logarithmic loss. It could be removed with a little technicality and only serves to simplify the proofs.\\
--Due to multiplicities, the generalization of our result to dimensions $d\geq 2$ is not clear since the spectral theory in higher dimensions becomes much more complicated. Therefore, it would certainly be necessary to work with Hermite multipliers. \vspace{0.3cm}

\textbf{Literature and related results.}\\   
Stability results over long periods have been established in both high and low regularity regime. A typical stability result has been proved in \cite{zbMATH05655526}, where the authors replaced the potential $V$ in equation (\ref{N1}) by a Hermite multiplier $M$ allowing them to deal with the frequencies $2j-1+m_j$ whose standard non-resonance condition was done\footnote{It turns out that most of the time the strong non-resonance condition introduced in Section \fcolorbox{red}{white}{\ref{sect89}} holds when the standard one is satisfied.} in \cite{zbMATH05655526}, and  this helped simplify the spectral analysis. More precisely, they considered the equation 
\begin{align}\tag{$\widetilde{\mathrm{NLS}}$}\label{eqfake}
i\partial_t u=(T+M)u + \vert u \vert^{2p}u, \end{align} and proved the following:
  Let $r\gg 1.$ For almost all $M$, there exists $s_0(r)\gg 1$ such that for all $s \geq s_0(r)$, there exists $\varepsilon_0>0$, if ${\vert\vert u^{(0)} \vert\vert_{\widehat{H}^s}}=\varepsilon <\varepsilon_0$, the unique solution of (\ref{eqfake}) satisfies\vspace{0.1cm}\\
       $$ \hspace{1.2cm}\vert t \vert < \varepsilon^{-r} \implies \left\{
\begin{array}{ll}
{\Vert u(t) \Vert}_{\widehat{H}^s} \lesssim {\Vert u(0) \Vert}_{\widehat{H}^s}\vspace{0.1cm}\\ 
\sum\limits_{j\in\mathbb{Z}}\langle j\rangle^{2s}\big\vert \vert u_j(t)\vert^2- \vert u_j(0)\vert^2  \big\vert\lesssim\varepsilon^{3}.
\end{array}
\right.$$
Hence, thanks to the smoothness of the solutions, they were able to control and obtain stability for all modes and not only finitely many.  The paper \cite{zbMATH06187752} generalises this result and proves a similar theorem for the operator $T+V(x)+M$ where $V$ belongs to the Schwartz class. Furthermore, \cite{zbMATH05153584} and \cite{zbMATH05117851} have proven stability results for Klein-Gordon and NLS equations on tori: given $r \gg 1$ arbitrarily large, there exists $s_0(r)\gg 1$ such that for sufficiently small initial data ${\Vert u(0) \Vert}_{H^s}:= \varepsilon$ with $s>s_0(r),$ no significant exchange of energy occurs before very long times $\abs{t} \leq \varepsilon^{-r}$, and the modulus of the Fourier modes is nearly conserved, i.e., $\vert u_n(t)\vert^2 \simeq \vert u_n(0)\vert^2.$ The main limitation of all these results lies in the requirement $s \geq s_0(r),$ which appears essential in their proofs (especially in addressing issues related to small divisors) and in similar results for dispersive Hamiltonian partial differential equations found in works such as \cite{zbMATH01336865, zbMATH05268759, zbMATH05223751, zbMATH07370990, zbMATH07202505, zbMATH07372984, zbMATH05251626, zbMATH05655526, zbMATH06405702, zbMATH06927918}. However, some numerical experiments suggest that this smoothness assumption is irrelevant and can be avoided and that $s_0(r)$ does not have to be very large (see for instance \cite{zbMATH05236720}). To this matter, the paper \cite{bernier2021birkhoff}, recently done by Bernier and Grébert, generates effective methods in order to lower the regularity and still obtain the stability result. They proved the almost global preservation of the low harmonic energies over very long times for Klein-Gordon equation and NLS with both Dirichlet and periodic boundary conditions in the energy space. The crucial key point was developing a Birkhoff Normal Form Theorem in low regularity which is weaker than the classical version of the theorem, since it only concerns the low modes of the solution. In \cite{bernier:hal-03334431} and along with Rivière, they extended this method to the sphere and worked with the Klein-Gordon equation.
The authors in \cite{zbMATH05968691} worked with a similar nonlinear Schr\"{o}dinger equation as \fcolorbox{red}{white}{\ref{N1}}, and constructed a class of potentials with the help of the dual basis of the finite family of Hermite polynomials ${(h_j^2)}_{1\leq j\leq n}.$ However, the new developed method we work with is simpler and applies to a larger class of potentials and initial data.
      
\subsection{Sketch of the proofs}
We will formally explain the strategies of the proofs. Concerning Theorem \fcolorbox{red}{white}{\ref{ikea1000}}, the method of the proof requires the Birkhoff normal form process introduced in \cite{bernier2021birkhoff}. Roughly speaking, the idea is to design or construct a symplectic\footnote{The symplectic transformations preserve the Hamiltonian structure.} and close to the identity map $\tau$ which helps simplify the Hamiltonian system. More precisely, composing with $\tau$, we push the non-normalized part of $H$ to higher orders and thus killing the terms that influence the dynamics of the low modes. For the sake of simplicity, we do the case $p=1.$ We denote by $\phi_{\chi}^t(u)$ the flow generated by $\chi$, a polynomial of degree 4, solving the equation $-i\partial_t \phi_{\chi}^t = (\nabla \chi)\circ \phi_{\chi}^t$. We write (\fcolorbox{red}{ white}{\ref{N1}}) as a Hamiltonian system\footnote{We give here a sketch of the proof. Actually, in Section \fcolorbox{red}{ white}{\ref{z4}}, we truncate the frequencies up to some level $M$ in order to consider the finite dimensional framework.} with \[ H= Z_2+P+\mathcal{O}(\Vert  u\Vert^6 )\] given explicitly in Section \fcolorbox{red}{ white}{\ref{z4}}, where $Z_2$ is a quadratic Hamiltonian associated with the linear part of the equation and depends only on the actions ${(I_j)}_{j\geq 1}$. Also, $P$ is a perturbation of order 4 belonging to a Hamiltonian class (refer to \fcolorbox{red}{white}{\ref{z2}}) and written as \[P(u) = \sum_{\substack{j,\ell\in (\mathbb{N}^{*})^2}}P_{j,\ell} u_{j_1} u_{j_{2}} \overline{u_{\ell_{1}}} \overline{u_{\ell_{2}}}.\]  As previously mentioned, we construct a symplectic close to the identity map $\tau$ such that, in the new variables, the Hamiltonian $H$ is a function of the actions up to a remainder $R$ of arbitrarily high order (we say that $H$ is written in a Birkhoff normal form). More precisely, as a first step we compose by $\phi_{\chi}^1$ and use Taylor expansion in order to get 
\[H\circ \phi_{\chi}^1 = Z_2 + \{\chi,Z_2\} +P+ \mathcal{O}(\Vert u \Vert^{6})\] where $\{ \cdot , \cdot \}$ denotes the Poisson brackets (refer to \fcolorbox{red}{ white}{\ref{notations}}). For the sake of normalisation and in order to eliminate the monomials $u_{j_1} u_{j_{2}} \overline{u_{\ell_{1}}} \overline{u_{\ell_{2}}}$ that do not depend on the actions, we aim to solve the cohomological equation
\[ \{\chi,Z_2\}+P = Q \quad \text{ with } \quad   \{Q, I_j\} = 0 \text{ for } j \leq N.\]
However, during the process of solving this equation, small divisors in the form of \[ w_{j_1}+w_{j_2} - w_{\ell_1}-w_{\ell_2}\] might appear in the denominator, with $w$ denoting the family of frequencies of the operator $T+V$. As a result, we consider the strong non-resonance condition characterised by controlling these small divisors from below. The next step, would be to iterate this construction, and compose with a new symplectic map. At the end, we obtain a transformation pushing the non-normalized part of $H$ to order 6, followed by a transformation pushing it to order 8 and so on. Consequently, we get \[ H \circ \tau = Z_2 +Q +R\] where $Q$ commutes with the low actions $I_{j}(u)$ for $j \leq N$ and $R$ satisfies the estimate
 \begin{align}\label{remainderterm}
 {\vert\vert\nabla R(u)\vert\vert}_{h^{-1/2}} &\lesssim_N {\vert\vert u\vert\vert}_{h^{1/2}}^{2r+2}.\end{align}
 As a corollary of this result, introducing a new variable $v=\tau^{-1}(u),$ we notice that $v$ is the solution of the equation \[i\partial_tv = \nabla \tilde{H}(v)\quad \text{with }\quad \tilde{H}(v) = H\circ \tau (v) = Z_2(v)+Q(v)+R(v).\]
 Furthermore, we work with $\partial_t I_j(v(t))$ to get
    \begin{align*}
       \partial_t I_j(v)&= {( \nabla I_j(v), \partial_t v)}_{\ell^2}
        = {( i\nabla I_j(v), \nabla (Z_2+Q+R)(v))}_{\ell^2}= \{ I_j,R\}(v).
    \end{align*}
Finally, we conclude the long time estimate (\fcolorbox{red}{ white}{\ref{mainresult}}) by duality estimates, the Mean Value Inequality, the estimate (\fcolorbox{red}{ white}{\ref{remainderterm}}) and the control of the $\widehat{H}^1$-norm of $u(t)$ with ${\Vert u^{(0)} \Vert}_{\widehat{H}^1} = \varepsilon$.\vspace{0.5cm}

Now, we turn to the proof of Theorem \fcolorbox{red}{white}{\ref{ber51}}. As explained above, when simplifying the Hamiltonian system, we face a problem of small divisors that can be solved by an effective control of the frequencies of operator $T+V$. To do this, we follow the ideas of \cite{bernier2021birkhoff}. We seek a control of the first derivative of the small divisors in the simple case where $V=0$ \textcolor{black}{as detailed in  Lemma \fcolorbox{red}{white}{\ref{lemimp8}}}. In order to proceed, it seemed necessary to control the norm of $V$ by $N^{-1/6}.$ Thus, the relation between $V$ and $N$ appears and consequently, we estimate the first derivative with respect to $V$ of the small divisors for $V\neq 0.$ Finally, using probability arguments, we deduce a control of the small divisors by the smallest index involved. Further tools of spectral analysis are needed to obtain the main non-resonant condition (for details see Section \fcolorbox{red}{white}{\ref{sect89}}).

\textbf{Organization of the article.} \textcolor{black}{In section 2, we work with spectral theory, where we provide a set of technical and non trivial tools to establish the non-resonance condition for the corresponding spectrum. Section 3 introduces a well-chosen class of Hamiltonian functions suitable for the nonlinear quantum harmonic oscillator and satisfying nice properties. Motivated by the work of \cite{bernier2021birkhoff}, section 4 is devoted to developing a normal form procedure in low regularity which plays an essential role in the work. In the last section, we present our main result, that demonstrates the almost global preservation of the low actions over very long times obtained as a a remarkable dynamical corollary of the Birkhoff normal form theorem.}

  \subsection{{Notations}}\label{notations}
    We always consider the following set of notations:
\begin{itemize}
\item We define for $s \geq 0$ the Sobolev spaces 
\begin{align}\label{Sobolev}
\widehat{H}^s &:= \{ u \in H^s(\mathbb{R}),\hspace{0.1cm} \langle x\rangle ^s u \in L^2(\mathbb{R})\} \quad \text{ with } \quad {\Vert u \Vert}_{\widehat{H}^s}^2 := {\Vert u \Vert}_{H^s}^2 + {\Vert  \langle x \rangle^s u \Vert}_{L^2}^2. 
\end{align}

\item $2\partial_{\overline{z}} := \partial_{\Re\hspace{0.05cm}z} + i \partial_{\Im\hspace{0.05cm}z}$ and $2\partial_{z} := \partial_{\Re \hspace{0.05cm}z} - i \partial_{\Im\hspace{0.05cm}z}.$
\item \textcolor{black}{For all $M \in (1,\infty]$, we define $\llbracket 1,M \rrbracket:=\{1, 2, \cdots, M-1, M \}$}
\item \textcolor{black}{$w \in \mathbb{R}^{\llbracket 1,M \rrbracket}$ denotes $ w \equiv (w_n)_{n\in {\llbracket 1,M \rrbracket}}$}.
\item For simplicity of notations, we write $x \lesssim_p y$ if there exists a constant $C$ depending on $p$ fixed such that $x\leq C y$ for $(x,y) \in \mathbb{R}^2$.
\item \textcolor{black}{$\mathscr{S}_r$ denotes the symmetric group of degree $r$}.
\item For $k \in \mathbb{Z}$, the Japanese bracket is denoted by $\langle k \rangle$ := $(1+\abs{k}^2)^{1/2}$.
\item For $s\in \mathbb{R}$ and $M >1$, the discrete Sobolev space is written as \[h^s(\llbracket 1, M \rrbracket) = \left\{u \in \mathbb{C}^{\llbracket 1, M \rrbracket },\hspace{0.1cm} {\Vert u \Vert}_{h^s}^2:=\sum_{k\in \llbracket 1, M \rrbracket} \langle k \rangle ^{2s}\abs{u_k}^2 < \infty \right\}.\]
\item For $p\geq 1$ and $M>1,$ the Lebesgue space is written as
\[ \ell^p(\llbracket 1, M \rrbracket) = \left\{u \in \mathbb{C}^{ \llbracket 1, M \rrbracket},\hspace{0.1cm} {\Vert u \Vert}_{\ell^p}^p:=\sum_{k\in \llbracket 1, M \rrbracket} \abs{u_k}^p < \infty \right\}.\]

\end{itemize}

\section{Non-resonance condition}\label{z1}

In this section, we describe first the spectrum of the operator $T+V$ for $V \in \mathscr{C}^2 \cap \widehat{H}^1$ in order to deduce that, almost surely, this spectrum is strongly non-resonant according to the definition given in \fcolorbox{red}{ white}{\ref{sect89}}. For spectral aspects, we start by considering the potential $V$ written in terms of the $L^2$-basis ${(h_k(\cdot \sqrt{2}) 2^{1/4})}_{k \geq 1}$ as 
\begin{align}\label{ikea1}
    V(x)&= 
    \sum_{k\geq 1}v_k h_k(x\sqrt{2}) 2^{1/4}.
    \end{align}
\textcolor{black}{At some point in Section \fcolorbox{red}{ white}{\ref{sect89}}, we will need to compute the integral $\int_{\mathbb{R}}2^{1/4} h_{2k-1}(x\sqrt{2})h_j^2(x)  \, \mathrm{d}x$ where $2^{1/4} h_{2k-1}(x\sqrt{2})$ appears from the decomposition of $V$. Since ${(h_j^2)}_{j\geq 1}$ can be expressed in terms of the basis ${(h_{2k-1}(\cdot \sqrt{2})2^{1/4})}_{k\geq 1}$ (refer to Lemma \fcolorbox{red}{ white}{\ref{jo1}}), then this integral can be easily computed by using the orthogonality property of the Hermite functions.}

\subsection{Preliminaries on spectral analysis}
In this part, we are interested in estimations on the Lebesgue norms of the eigenfunctions ${(\psi_j)}_{j\geq 1}$ of the operator $T+V$. For $V \in L^{\infty}(\mathbb{R})$, we have that $T+V$ is a self-adjoint operator of domain $\widehat{H}^2$ (refer to Chapter 2 in \cite{zbMATH07312855}) and has a real, discrete spectrum (refer to \cite{lewin:cel-01935749}), consisting of simple eigenvalues ${(\Lambda_j)}_{j\geq 1}$ and satisfying \[\Lambda_j = 2j-1 +\smallO(1), \hspace{0.1cm}j \rightarrow \infty.\] As a consequence, similar to ${(h_j)}_{j\geq 1}$, the eigenfunctions ${(\psi_j)}_{j\geq 1}$ of the operator $T+V$ form an orthonormal basis of $L^2(\mathbb{R})$. \textcolor{black}{In other words, the following are satisfied: The $L^2-$normalisation property (${\Vert \psi_j \Vert}_{L^2}=1$) and the orthogonality property  (${(\psi_j, \psi_\ell)}_{L^2} = \delta_{j,\ell})$}, and thus we are able to ensure the spectral decomposition. We end up by estimating the corresponding eigenvalues ${(\Lambda_j)}_{j\geq 1}$ of the operator $T+V$. Using an important result of Koch--Tataru \cite{zbMATH02201105}, we obtain the following lemma which is the key to our estimations: 
\begin{lem}\label{N6}
For all $j\geq 1$ and $V \in \mathscr{C}^2 \cap \widehat{H}^1,$ there exists $C>0$ such that
\begin{align*}
{\Vert \psi_j\Vert}_{L^{4}} &\leq C j ^{-1/12}.
\end{align*}
\end{lem}

\begin{proof}
Applying H\"older's inequality, we have the estimation
\begin{align} \label{er1}
{\Vert \psi_j\Vert}_{L^4} &\leq {\Vert \psi_j\Vert}_{L^6}^{3/4}{\Vert \psi_j\Vert}_{L^2}^{1/4}.
\end{align}
Applying Corollary 3.2 from \cite{zbMATH02201105} for $W(x) = x^2+V(x)$ and $p=6,$ we obtain \footnote{A personal communication by Herbert Koch regarding Theorem 4 in \cite{zbMATH02201105}: The proof can be modified in order to deal with $W \in \mathscr{C}^2.$}
\[{\Vert \psi_j \Vert}_{L^6} \lesssim j^{-1/9}{\Vert \psi_j \Vert}_{L^2}.\]
Thus replacing in (\fcolorbox{red}{ white}{\ref{er1}}) and using that ${(\psi_j)}_{j\geq 1}$ is an orthonormal basis, we get
\[{\Vert \psi_j\Vert}_{L^4} \lesssim  \left(j^{-1/9}{\Vert \psi_j \Vert}_{L^ 2}\right) ^{3/4}{\Vert \psi_j\Vert}_{L^2}^{1/4} \lesssim j^{-1/12}. \qedhere\] 
\end{proof}

\begin{RQ}
Note that for the case $W(x)=x^2$ \hspace{0.1cm} (i.e. $V=0),$ the norm ${\Vert h_j \Vert}_{L^{4}}$ can be easily estimated by using Lemma \fcolorbox{red}{ white}{\ref{jo1}} and Parseval--Bessel's equality.
\end{RQ}

\textbf{Notations.} In the following three results, we denote by $\Lambda_{j,V}$ (resp. $\psi_{j,V}$) the eigenvalues (resp. eigenfunctions) of the operator $T+V$. We adapt the proofs done in \cite{zbMATH05968691}.\vspace{0.3cm}\\
In this lemma, we can see that the eigenvalues are close to integer values.
\begin{lem}\label{N7}
 For all $j\geq 1$ and $V \in \widehat{H}^1$ small enough with respect to the norm ${\Vert \cdot \Vert}_{\widehat{H}^1}$, we have $$\abs{\Lambda_{j,V}- ( 2j-1)} \lesssim {\Vert V \Vert}_{\widehat{H}^1} j^{-1/2}.$$
\end{lem}

\begin{proof}
We refer the reader to Lemma 2.1 \textcolor{black}{and Lemma 2.3} in \cite{zbMATH02243649}.
\end{proof}

The next lemma serves as a useful tool for Proposition \fcolorbox{red}{ white}{\ref{N80}}.
\begin{lem}\label{pl1}
   For $V_1, V_2 \in \widehat{H}^1 \cap \mathscr{C}^2$ small enough with respect to the norm ${\Vert \cdot \Vert}_{\widehat{H}^1}$, there exists $C>0$ such that for all $j\geq 1$
   \[{\Vert \psi_{j,V_2}-{(\psi_{j,V_1},\psi_{j,V_2})}_{L^2} \psi_{j,V_1} \Vert}_{L^2} \leq Cj^{-1/12} {\Vert V_1-V_2\Vert}_{\widehat{H}^1}.\]
\end{lem}

\begin{proof}
Since ${(\psi_k)}_{k\geq 1}$ is a Hilbertian basis of $L^2(\mathbb{R}),$ then it is natural to decompose 
\begin{align}
  \nonumber  {\Vert \psi_{j,V_2}-{(\psi_{j,V_1},\psi_{j,V_2})}_{L^2} \psi_{j,V_1} \Vert}_{L^2}^2
    &=  \nonumber \sum_{k\geq 1}\abs{{\left(\psi_{j,V_2}-{(\psi_{j,V_1},\psi_{j,V_2})}_{L^2} \psi_{j,V_1}, \psi_{k,V_1}\right)}_{L^2}}^2\\
    &=  \nonumber \sum_{k\geq 1}\big\vert{(\psi_{j,V_2},\psi_{k,V_1})}_{L^2}-{(\psi_{j,V_1},\psi_{j,V_2})}_{L^2} {(\psi_{j,V_1}, \psi_{k,V_1})}_{L^2}\big\vert^2\\
    &= \label{pv}\sum_{\substack{k\geq 1 \\ k\neq j}}\abs{ {(\psi_{j,V_2}, \psi_{k,V_1})}_{L^2}}^2
\end{align}
because ${{(\psi_{j,V_1}, \psi_{k,V_1})}_{L^2}} = \delta_{k,j}.$ 
Similarly, since $T+V_1-\Lambda_{j,V_2}$ is self-adjoint we write 
\begin{align*}
 {\Vert (T+V_1-\Lambda_{j,V_2})\psi_{j,V_2}\Vert}_{L^2}^2
 &= \sum_{k\geq 1}\abs{{\left((T+V_1-\Lambda_{j,V_2})\psi_{j,V_2}, \psi_{k,V_1} \right)}_{L^2}}^2\\
 &= \sum_{k\geq 1}\abs{{\left((T+V_1-\Lambda_{j,V_2})\psi_{k,V_1}, \psi_{j,V_2} \right)}_{L^2}}^2\\
  &= \sum_{k\geq 1}\abs{{\left((\Lambda_{k,V_1}-\Lambda_{j,V_2})\psi_{k,V_1}, \psi_{j,V_2} \right)}_{L^2}}^2\\
  &=  \sum_{k\geq 1}\abs{\Lambda_{k,V_1}-\Lambda_{j,V_2}}^2\abs{{(\psi_{k,V_1}, \psi_{j,V_2})}_{L^2}}^2.
 \end{align*} 
Now from Lemma \fcolorbox{red}{ white}{\ref{N7}}, we have that
\begin{align}\label{simestimate}
\abs{\Lambda_{k,V_1}-\Lambda_{j,V_2}} &\gtrsim1 \end{align}
for $k \neq j$ uniformly in $V_1, V_2$ small enough with respect to ${\Vert \cdot \Vert}_{\widehat{H}^1}$. \textcolor{black}{Indeed, assume that $k>j$ (the case $k<j$ is treated similarly), then we write
\begin{align*}
\Lambda_{k,V_1}-\Lambda_{j,V_2} &= 2(k-j) + \mathcal{O}({\Vert V_1 \Vert}_{\widehat{H}^1}k^{-1/2})-\mathcal{O}({\Vert V_2 \Vert}_{\widehat{H}^1}j^{-1/2})\\
& \geq 2(k-j) -C_1{\Vert V_1 \Vert}_{\widehat{H}^1}k^{-1/2} - C_2 {\Vert V_2 \Vert}_{\widehat{H}^1}j^{-1/2}\\
& \geq 2(k-j) - C_3 \min({\Vert V_1 \Vert}_{\widehat{H}^1}, {\Vert V_2 \Vert}_{\widehat{H}^1}) (k^{-1/2} +j^{-1/2} )\\
& \geq 2(k-j) - 2C_3 \min({\Vert V_1 \Vert}_{\widehat{H}^1}, {\Vert V_2 \Vert}_{\widehat{H}^1})
\end{align*}
implying the result provided that ${\Vert V_1 \Vert}_{\widehat{H}^1}$ and  ${\Vert V_2 \Vert}_{\widehat{H}^1}$ are small enough.} Thus, (\fcolorbox{red}{ white}{\ref{simestimate}}) gives
\begin{align}\label{pa}
    {\Vert (T+V_1-\Lambda_{j,V_2})\psi_{j,V_2}\Vert}_{L^2}^2&\gtrsim  \sum_{\substack{k\geq 1 \\ k\neq j}}\abs{{(\psi_{k,V_1}, \psi_{j,V_2})}_{L^2}}^2.
\end{align}
After this, applying (\fcolorbox{red}{ white}{\ref{pv}}) and (\fcolorbox{red}{ white}{\ref{pa}}) we deduce that
\begin{align*}
    {\Vert \psi_{j,V_2}-{(\psi_{j,V_1},\psi_{j,V_2})}_{L^2} \psi_{j,V_1} \Vert}_{L^2}^2 &=  \sum_{\substack{k\geq 1 \\ k\neq j}}\abs{{(\psi_{k,V_1}, \psi_{j,V_2})}_{L^2}}^2 \lesssim  {\Vert (T+V_1-\Lambda_{j,V_2})\psi_{j,V_2}\Vert}_{L^2}^2.
\end{align*}
Notice that we can write
\begin{align*}
    (T+V_1)\psi_{j,V_2}(x)&= (-\partial_{xx}+x^2+V_1)\psi_{j,V_2}(x)+(V_1-V_2)\psi_{j,V_2}(x)\\
    &= (-\partial_{xx}+x^2+V_2)\psi_{j,V_2}(x)+(V_1-V_2)\psi_{j,V_2}(x)\\
    &= \Lambda_{j,V_2}\psi_{j,V_2}(x)+(V_1-V_2)\psi_{j,V_2}(x),
\end{align*}
and H\"older's inequality implies that
\begin{align*}
    {\Vert (T+V_1- \Lambda_{j,V_2})\psi_{j,V_2}\Vert}_{L^2}&={\Vert (V_1-V_2)\psi_{j,V_2}\Vert}_{L^2} \leq {\Vert V_1-V_2\Vert}_{L^4}{\Vert \psi_{j,V_2}\Vert}_{L^4}.
\end{align*}
Next, using the Sobolev embedding $H^1 \hookrightarrow L^4$, the continuous inclusion $\widehat{H}^1 \subset H^1$  as well as Lemma \fcolorbox{red}{ white}{\ref{N6}}, we get
\begin{align} 
\nonumber{\Vert (T+V_1- \Lambda_{j,V_2})\psi_{j,V_2}\Vert}_{L^2}&\leq {\Vert V_1-V_2\Vert}_{H^1}{\Vert \psi_{j,V_2}\Vert}_{L^4}\\
&\leq\nonumber {\Vert V_1-V_2\Vert}_{\widehat{H}^1}{\Vert \psi_{j,V_2}\Vert}_{L^4}\\
&\leq \nonumber \label{ER1}C {\Vert V_1-V_2\Vert}_{\widehat{H}^1} j^{-1/12}. 
\end{align}

\end{proof}
We prove now that the eigenfunctions ${(\psi_j)}_{j \geq 1}$ are close to the Hermite functions.
\begin{prop}\label{N80}
For all $j\geq 1$ and $V \in \widehat{H}^1 \cap \mathscr{C}^2$ small enough with respect to the norm ${\Vert \cdot \Vert}_{\widehat{H}^1}$, there exists $C>0$ such that
 \[{\Vert \textcolor{black}{\psi_{j,V}} - h_j\Vert}_{L^2} \leq Cj^{-1/12}{\Vert V \Vert}_{\widehat{H}^1}.\]
\end{prop}

\begin{proof}
Taking the scalar product of $\psi_{j,V_2}-{(\psi_{j,V_1},\psi_{j,V_2})}_{L^2}\psi_{j,V_1}$ with $\psi_{j,V_2}$, we get
\begin{align*}
    \abs{{\left(\psi_{j,V_2},\psi_{j,V_2}-{(\psi_{j,V_1},\psi_{j,V_2})}_{L^2} \psi_{j,V_1}\right)}_{L^2}}
    &= \abs{1-{(\psi_{j,V_1},\psi_{j,V_2})}_{L^2}^2}. 
\end{align*}
Therefore, applying Cauchy--Schwarz inequality and Lemma \fcolorbox{red}{ white}{\ref{pl1}} we have
\begin{align}
   \abs{1-{(\psi_{j,V_1},\psi_{j,V_2})}_{L^2}^2}
    &\leq {\Vert \psi_{j,V_2} \Vert}_{L^2}  {\Vert \psi_{j,V_2}-{(\psi_{j,V_1} ,\psi_{j,V_2})}_{L^2} \psi_{j,V_1} \Vert}_{L^2} \lesssim \label{bb1}j^{-1/12}  {\Vert V_1-V_2\Vert}_{\widehat{H}^1}.
\end{align}
Finally, note that adding the terms ±${(\psi_{j,V_1},\psi_{j,V_2})}_{L^2} \psi_{j,V_1}$ gives
\begin{align*}
    {\Vert \psi_{j,V_1} - \psi_{j,V_2}\Vert}_{L^2}^2
    &\leq 2{\Vert \psi_{j,V_2}- {(\psi_{j,V_1},\psi_{j,V_2})}_{L^2} \psi_{j,V_1}\Vert}_{L^2}^2 + 2{\Vert \psi_{j,V_1}\left(1- {(\psi_{j,V_1},\psi_{j,V_2})}_{L^2}\right)\Vert}_{L^2}^2\\
    &= 2{\Vert \psi_{j,V_2}- {(\psi_{j,V_1},\psi_{j,V_2})}_{L^2} \psi_{j,V_1}\Vert}_{L^2}^2 + 2\abs{1- {(\psi_{j,V_1},\psi_{j,V_2})}_{L^2}}^2\underbrace{{\Vert \psi_{j,V_1}\Vert}_{L^2}^2}_{=1}.
\end{align*}
Hence, using Lemma \fcolorbox{red}{ white}{\ref{pl1}} and (\fcolorbox{red}{ white}{\ref{bb1}}), we obtain 
\begin{align*}
    {\Vert \psi_{j,V_1} - \psi_{j,V_2}\Vert}_{L^2}^2
    &\leq 2{\Vert \psi_{j,V_2}- {(\psi_{j,V_1},\psi_{j,V_2})}_{L^2} \psi_{j,V_1}\Vert}_{L^2}^2 + 2\abs{1- {(\psi_{j,V_1},\psi_{j,V_2})}_{L^2}^2}^2\\
    &\lesssim 2(j^{-1/12} {\Vert V_1-V_2\Vert}_{\widehat{H}^1})^2 + 2(j^{-1/12} {\Vert V_1-V_2\Vert}_{\widehat{H}^1})^2\\
    &\lesssim (j^{-1/12} {\Vert V_1-V_2\Vert}_{\widehat{H}^1})^2.
    \end{align*}
In particular, for $V_2=0$ we have $\psi_{j,V_2}(x)=h_j(x)$ and thus the needed result.
\end{proof}

Finally, using the expression of $V$ and the expansion of $h_j^2,$ we get the following result:

\begin{prop}\label{N200}
For $V \in \widehat{H}^1$, the gradient of the eigenvalues $\Lambda_j(V)$ with respect to $v_{2k-1}$ (recall that $v_k$ are the coefficients from the expansion of the potential $V$ in (\fcolorbox{red}{ white}{\ref{ikea1}})) is given by
 \[ \partial_{v_{2k-1}}\Lambda_j(V)= \int_{\mathbb{R}}2^{1/4} h_{2k-1}(x\sqrt{2})\psi_j^2(x)  \, \mathrm{d}x.\]
\end{prop}

\begin{RQ}
\textcolor{black}{We switch to index $2k-1$ to bring forth the term $h_{2k-1}$ needed later to simplify the computations of the integral $\int_{\mathbb{R}}2^{1/4} h_{2k-1}(x\sqrt{2})h_j^2(x)  \, \mathrm{d}x$, as mentioned in the beginning of Section \fcolorbox{red}{ white}{\ref{z1}}.}
\end{RQ}

\begin{proof}
\textcolor{black}{For the proof, we refer the reader to Lemma 2.4 in \cite{zbMATH02243649}. In addition, we present formal computations where we assume that each eigenvalue $\Lambda_j$ and each eigenfunction $\psi_j$ for $j\geq 1$ is $\mathscr{C}^1$ with respect to $v_{2k-1}$.} We consider the equation
\[(\underbrace{-\partial_{xx}+x^2+V(x)}_{T+V})\psi_j(x) =\Lambda_j\psi_j(x).\] Differentiating the above with respect to $v_{2k-1}$ for $k\geq 1$ we obtain
\[(T+V)\frac{\partial \psi_j}{\partial v_{2k-1}} + \frac{\partial (T+V)}{\partial v_{2k-1}}\psi_j = \frac{\partial \Lambda_j}{\partial v_{2k-1}}\psi_j + \Lambda_j\frac{\partial \psi_j}{\partial v_{2k-1}}.\]
Due to the expression of $V$ given by (\fcolorbox{red}{ white}{\ref{ikea1}}), this implies that
\[(T+V-\Lambda_j)\frac{\partial \psi_j}{\partial v_{2k-1}} + 2^{1/4} h_{2k-1}(\cdot\sqrt{2})\psi_j = \frac{\partial \Lambda_j}{\partial v_{2k-1}}\psi_j.\] 
Next, taking the scalar product with $\psi_j$ we get
\begin{equation}\label{M15.}
\left((T+V-\Lambda_j)\frac{\partial \psi_j}{\partial v_{2k-1}},\psi_j\right)_{L^2} + 2^{1/4}{(h_{2k-1}(\cdot\sqrt{2})\psi_j,\psi_j)}_{L^2} = {\left(\frac{\partial \Lambda_j}{\partial v_{2k-1}}\psi_j,\psi_j\right)} _{L^2}.\end{equation}
Using self-adjointness of $T+V$ and the fact that $\psi_j \in \text{ker}(T+V-\Lambda_j)$, we deduce 
\[{\left((T+V-\Lambda_j)\frac{\partial \psi_j}{\partial v_{2k-1}},\psi_j\right)}_{L^2} = {\left(\frac{\partial \psi_j}{\partial v_{2k-1}},(T+V-\Lambda_j)\psi_j\right)}_{L^2} =0.\]
Since $\frac{\partial \Lambda_j}{\partial v_{2k-1}}$ is independent of $x$ and ${\Vert\psi_j\Vert}_{L^2} =1,$ then (\fcolorbox{red}{ white}{\ref{M15.}}) gives
\begin{align}\label{imp0}
    \partial_{v_{2k-1}}\Lambda_j(V)&= \int_{\mathbb{R}}2^{1/4} h_{2k-1}(x\sqrt{2})\psi_j^2(x)  \, \mathrm{d}x.
    \end{align}
\end{proof}



 \subsection{Non-resonance condition}\label{sect89}  
   In the second part, we are interested in probabilistic aspects. For this, given a weight $P \in \widehat{H}^3$ such that $P_k \in \mathbb{R}_+^*$, we draw $V$ randomly (recall (\fcolorbox{red}{ white}{\ref{ber50}})) as
\begin{align}\label{ikea998}
    V(x)&= \sum_{k\geq 1} g_kh_k(x\sqrt{2}) P_k
    \end{align}
where $g_k \sim \mathcal{N}(0,1)$ are some independent Gaussian variables. It is important to emphasize that adding such a weight ensures the following technical assumptions\footnote{These assumptions are used to prove Proposition \fcolorbox{red}{ white}{\ref{ikea5}}.} (for the proof of (\fcolorbox{red}{ white}{\ref{my1}}), refer to Lemma \fcolorbox{red}{ white}{\ref{lems}}) on $V$:
\begin{align}\label{my1}\left\{
    \begin{array}{ll}
 V \in \widehat{H}^1 \cap \mathscr{C}^2     \text{ almost surely}, \\
        \mathbb{P}({\Vert V \Vert}_{\widehat{H}^1} < \lambda )>0 \text{ for all } \lambda>0.
    \end{array}
\right.\end{align}
 We imitate the work done in \cite{bernier2021birkhoff} to prove that the frequencies of (\fcolorbox{red}{ white}{\ref{N1}}) (also known as the eigenvalues $\Lambda_j$ of the operator T+V) obtained from the quadratic Hamiltonian are strongly $N,r$ non-resonant \textcolor{black}{in the sense of Definition \fcolorbox{red}{ white}{\ref{dep1}}.} To prove this condition we use the following tool taken from \cite{bernier2021birkhoff}:
   
\begin{prop}\label{pr55}
Let $r \geq 1$, $N\geq 1$ and $w\in \mathbb{R}^{\mathbb{N}^*}$. Suppose that:
\begin{itemize}
\item[i)] the frequencies are weakly non-resonant, i.e. for all $1\leq r^* \leq r$, there exist $\alpha_{r^*}>0$ and $\gamma_{r,N}>0$ such that for all $\sigma \in (\mathbb{Z}^*)^{r^*}$ and all $j\in (\mathbb{N}^*)^{r^*}$ with $ j_1 < \cdots < j_{r^*} $, $j_1\leq N$ and $\abs{\sigma_1}+\cdots+\abs{\sigma_{r^*}} \leq r$, we have 
\begin{align}\label{eq7} \forall k\in \mathbb{Z},\hspace{0.3cm} \abs{k+ \sigma_1w_{j_1}+ \cdots + \sigma_{r}w_{j_{r^*}}} {\geq} \hspace{0.1cm} \gamma_{r,N} j_{r^*}^{-\alpha_{r^*}},\end{align}
\item [ii)] the frequencies accumulate polynomially fast on $\mathbb{Z}$, i.e. there exists $C>0$ and $\textcolor{black}{a}>0$ such that 
\begin{align}\label{ez3} \forall j \geq 1, \exists \, k\in \mathbb{Z},\quad \abs{w_j - k}&\leq C  j ^{-a}. 
\end{align}
\end{itemize}
Then $w$ is strongly $N,r$ non-resonant.
\end{prop}

\begin{RQ}It is important to mention that the first assumption is satisfied by many interesting Hamiltonians, for instance Beam and Klein-Gordon equations. However, the localization assumption is easier to check but seems to be more restrictive.
\end{RQ}

\begin{proof}
The proof is done by induction on $r^*$ and is found in Proposition 2.1 of \cite{bernier2021birkhoff}. 
\end{proof}
Our goal now is to apply Proposition \fcolorbox{red}{ white}{\ref{pr55}} and obtain the main result of this section, Proposition \fcolorbox{red}{ white}{\ref{po1}}. To do so, we concentrate in what follows on proving that the frequencies ${(\Lambda_j)}_{j\geq 1}$ satisfy the weak non-resonance condition. We start with some useful lemmas. In the first one, we express $h_j^2$ in terms of the Hilbertian basis ${(h_{2k-1}(\cdot\sqrt{2})2^{1/4})}_{k\geq 1}$. The process was inspired by the decomposition of the product of the Hermite functions $h_j(x)h_l(x)$. In the case where $j=l,$ we obtain the result given as Proposition 5.5 in \cite{zbMATH06187752} \textcolor{black}{with a small change of indices where the sequence ${(h_{j}^2)}_{j\geq 1}$ here corresponds to the sequence ${(h_{n}^2)}_{n\geq 0}$ in \cite{zbMATH06187752}.}
\begin{lem}\label{jo1}
For all $j\geq 1$, we can write
\[h_j^2(x) = \sum_{k=1}^{j} \mu_{k,j} h_{2k-1}(x\sqrt{2})2^{1/4}\]
  with \[\mu_{k,j}= (2\pi)^{-1/4}\sqrt{\alpha_{k}}\alpha_{j-k+1} \quad \text{and} \quad \alpha_j=\frac{(2j-2)!}{(j-1)!^2 4^{j-1}} \sim \frac{1}{\sqrt{\pi j}}.\]
\end{lem}

It is easy to establish bounds for this explicit form. \textcolor{black}{Note that the constants in the following two inequalities may not be the same. For the sake of simplicity, we denote them by $C$.}
\begin{cor}\label{impcor}
 There exists $C>0$ such that for all $j\geq 1$ and $k\leq j,$ we have
 \[\mu_{k,j} \leq Cj^{-1/4}.\] Moreover, for $j=k$ we also have the lower bound \[\mu_{j,j}\geq C^{-1}j^{-1/4}.\]
\end{cor}

\begin{proof}
We write for $j\geq 1$ and $k\leq j$, $\mu_{k,j}= (2\pi)^{-1/4}\sqrt{\alpha_{k}}\alpha_{j-k+1} \leq \frac{C}{k^{1/4}(j-k)^{1/2}}.$
\begin{itemize}
    \item If $k\geq j/2,$ then $\mu_{k,j} \leq 2^{1/4}C j^{-1/4}$,
    \item If $k\leq j/2,$ then $\mu_{k,j} \leq  2^{1/2}C j^{-1/2}\leq Cj^{-1/4}.$
    \end{itemize}
    Moreover, by definition of $\alpha_j$, we naturally have $\mu_{j,j} = (2\pi)^{-1/4}\sqrt{\alpha_{j}} \sim \frac{1}{(\pi j)^{1/4}}.$
\end{proof}

We are interested now in deducing an estimation on the derivative for $V \neq 0.$
\begin{lem}\label{lemimp2}
For all $\rho >0,$ there exists $C>0$ such that for all $j\geq 1,$ $k\geq 1$ and ${\Vert V \Vert}_{\widehat{H}^1} \leq \rho,$ we have 
\[\abs{ \partial_{v_{2k-1}}\Lambda_j(V)-\mu_{k,j}} \leq Cj^{-1/12}{\Vert V\Vert}_{\widehat{H}^1}.\]
\end{lem}

\begin{proof}
From (\fcolorbox{red}{ white}{\ref{imp0}}) we have 
\[ \partial_{v_{2k-1}}\Lambda_j(V)= \int_{\mathbb{R}}2^{1/4} h_{2k-1}(x\sqrt{2})\psi_j^2(x)  \, \mathrm{d}x\]
and in particular by the decomposition from Lemma \fcolorbox{red}{ white}{\ref{jo1}} \[ \partial_{v_{2k-1}}\Lambda_j(0)= \int_{\mathbb{R}}2^{1/4} h_{2k-1}(x\sqrt{2})h_j^2(x)  \, \mathrm{d}x = \mu_{k,j}.\] 

Furthermore, using Propostion \fcolorbox{red}{ white}{\ref{N80}} we have
\begin{align*}
    \abs{\int_{\mathbb{R}} h_{2k-1}(\psi_j^2-h_j^2) \, \mathrm{d}x} &\leq {\Vert h_{2k-1}\Vert}_{L^{\infty}}{\Vert \psi_j-h_j\Vert}_{L^2}{\Vert \psi_j+h_j\Vert}_{L^2}\\
    &\leq {\Vert h_{2k-1}\Vert}_{\widehat{H}^1}{\Vert \psi_j-h_j\Vert}_{L^2}({\Vert \psi_j\Vert}_{L^2}+{\Vert h_j\Vert}_{L^2})\\
   &\leq C {\Vert \psi_j-h_j\Vert}_{L^2}\\
   &\leq Cj^{-1/12} {\Vert V\Vert}_{\widehat{H}^1}.
\end{align*}
Consequently, we easily deduce the needed estimation
\[ \abs{\int_{\mathbb{R}}2^{1/4} h_{2k-1}(x\sqrt{2})\psi_j^2(x)  \, \mathrm{d}x - \int_{\mathbb{R}}2^{1/4} h_{2k-1}(x\sqrt{2})h_j^2(x)  \, \mathrm{d}x} \leq  Cj^{-1/12}{\Vert V \Vert}_{\widehat{H}^1}. \qedhere \]
\end{proof}

\textbf{Notations.} We denote the small divisors by $\Omega_{j,\sigma}(V) = \sum\limits_{n=1}^{r^*}\sigma_n \Lambda_{j_n}\textcolor{black}{(V)}.$\vspace{0.3cm}\\
The last part of this section is inspired by the work done for NLS defined on $\mathbb{T}$ in \cite{bernier2021birkhoff}.
\begin{lem}\label{lemimp8}
For all $1\leq r^* \leq r,$ there exists $\gamma_{r}>0$ such that for all $\sigma \in (\mathbb{Z}^*)^{r^*}$ and $j\in (\mathbb{N}^*)^{r^*}$ with $ j_1 < \cdots < j_{r^*} $ and $\abs{\sigma_1}+\cdots+\abs{\sigma_{r^*}} \leq r$, there exists $k\lesssim_{r} j_1$ such that 
\begin{align}\label{lemimp}
    \abs{\partial_{v_{2k-1}} \Omega_{j,\sigma}(0)} &= \abs{\sum_{n=1}^{r^*}\sigma_n \mu_{k,j_n}}\geq \gamma_{r}j_1^{-1/4}.
    \end{align}
\end{lem}

\begin{proof}
Fix $r\geq 1.$ We proceed with the proof by induction on $r^*.$\\
\underline{Initial Step:} If $r^*=1,$ then for all $j_1\in \mathbb{N}^*$ we have by Corollary \fcolorbox{red}{ white}{\ref{impcor}}
\[ \abs{\partial_{v_{2j_1-1}} \Omega_{j,\sigma}(0)} = \abs{\sigma_1\mu_{j_1,j_1}}\gtrsim \frac{\abs{\sigma_1}}{j_1^{1/4}}\gtrsim j_1^{-1/4}. \]
\underline{Induction Step:} Assume that the result holds for all $1\leq r^*< r$, and we prove it for $r^*+1$. Let $\sigma \in (\mathbb{Z}^*)^{r^*+1}$ and $j \in (\mathbb{N}^*)^{r^*+1}$ be some indices satisfying $\abs{\sigma_1}+\cdots+\abs{\sigma_{r^*+1}} \leq r$ and $j_1< \cdots < j_{r^*+1}$ and suppose that there exists $k\leq C_{r}j_1$ such that (\fcolorbox{red}{ white}{\ref{lemimp}}) holds.
\begin{itemize}
    \item By Corollary \fcolorbox{red}{ white}{\ref{impcor}}, induction hypothesis and the fact that $\abs{\sigma_{r^*+1}}\leq r$, we have 
    \begin{align*}
        \abs{\sum_{n=1}^{r^*+1}\sigma_n \mu_{k,j_n}}= \abs{\sum_{n=1}^{r^*}\sigma_n \mu_{k,j_n}+\sigma_{r^*+1} \mu_{k,j_{r^*+1}}} &\geq \abs{\sum_{n=1}^{r^*}\sigma_n \mu_{k,j_n}}-\abs{\sigma_{r^*+1} \mu_{k,j_{r^*+1}}} \\
        &\geq \gamma_{r}j_1^{-1/4}-rCj_{r^*+1}^{-1/4}.
    \end{align*}
    Hence, if we take $j_{r^*+1} >(2rC\gamma_{r}^{-1})^4j_1$ we directly conclude that
    \[\abs{\sum_{n=1}^{r^*+1}\sigma_n \mu_{k,j_n}} \geq \frac{\gamma_{r}}{2}j_1^{-1/4}.\]
    \item Now if $j_{r^*+1} \leq(2rC\gamma_{r}^{-1})^4j_1,$ we consider $\partial_{v_{2j_{r^*+1}-1}} \Omega_{j,\sigma}(0) = \sigma_{r^*+1}\mu_{j_{r^*+1},j_{r^*+1}}.$ Consequently, we obtain by Corollary \fcolorbox{red}{ white}{\ref{impcor}} the result for $k=j_{r^*+1}$ and $\tilde{\gamma}_{r} = \frac{\gamma_{r}}{2r}$
    \[\abs{\partial_{v_{2j_{r^*+1}-1}} \Omega_{j,\sigma}(0)} \geq Cj_{r^*+1}^{-1/4} \geq  \tilde{\gamma}_{r} j_1^{-1/4}. \qedhere\]
\end{itemize}
\end{proof}

After this, we obtain a similar estimation for $V\neq 0.$
\begin{cor}\label{ikea2}
 For all $1\leq r^*\leq r$ and all $\sigma \in (\mathbb{Z}^*)^{r^*}$, $j\in (\mathbb{N}^*)^{r^*}$ and $V\in \widehat{H}^1$ satisfying $ j_1 < \cdots < j_{r^*},$ $\abs{\sigma}_1 \leq r$  and ${\Vert V \Vert}_{\widehat{H}^1}\leq \frac{\gamma_{r}}{2r}j_1^{-1/6}$ where $\gamma_r$ is given by Lemma \fcolorbox{red}{ white}{\ref{lemimp8}}, there exists $k \lesssim_{r} j_1$ such that 
 \[\abs{\partial_{v_{2k-1}} \Omega_{j,\sigma}(V)} \geq \frac{\gamma_{r}}{2}j_1^{-1/4}.\]
\end{cor}

\begin{proof}
From Lemma \fcolorbox{red}{ white}{\ref{lemimp2}}, we have that 
\begin{align*}
    \abs{\partial_{v_{2k-1}} \left(\sum\limits_{n=1}^{r^*}\sigma_n\Lambda_{j_n}(V)\right) - \partial_{v_{2k-1}} \left(\sum\limits_{n=1}^{r^*}\sigma_n\Lambda_{j_n}(0)\right)} &\lesssim \frac{{\Vert V\Vert}_{\widehat{H}^1}}{j_1^{1/12}}\abs{\sigma_1} +\cdots + \frac{{\Vert V\Vert}_{\widehat{H}^1}}{j_{r^*}^{1/12}}\abs{\sigma_{r^*}} \lesssim \frac{r {\Vert V\Vert}_{\widehat{H}^1}}{ j_1^{1/12}}.
\end{align*}
Thus, using Lemma \fcolorbox{red}{ white}{\ref{lemimp8}} and the assumption ${\Vert V \Vert}_{\widehat{H}^1}\leq \frac{\gamma_{r}}{2r}j_1^{-1/6},$ we establish
\[ \abs{\partial_{v_{2k-1}} \Omega_{j,\sigma}(V)}\geq \abs{\partial_{v_{2k-1}} \Omega_{j,\sigma}(0) - \frac{r {\Vert V\Vert}_{\widehat{H}^1}}{ j_1^{1/12}}} \geq \abs{\partial_{v_{2k-1}} \Omega_{j,\sigma}(0)}- \frac{r {\Vert V\Vert}_{\widehat{H}^1}}{ j_1^{1/12}}\geq \frac{\gamma_{r}}{j_1^{1/4}} - \frac{\gamma_{r}}{2j_1^{1/4}} \geq \frac{\gamma_{r}}{2j_1^{1/4}}. \qedhere\]
\end{proof}

As a result, we obtain the necessary weak non-resonance condition presented in the next Proposition. Recall that here we are considering $V$ as random potentials given in (\fcolorbox{red}{ white}{\ref{ikea998}}).
\begin{prop}\label{ikea5}
For all $1 \leq r^* \leq r$ and $N\geq 1,$ provided that ${\Vert V \Vert}_{\widehat{H}^1}\leq \frac{\gamma_{r}}{2r}N^{-1/6}$ where $\gamma_r$ is given by Lemma \fcolorbox{red}{ white}{\ref{lemimp8}}, almost surely, there exists $\gamma_{r,N}>0$ such that for all $\sigma \in (\mathbb{Z}^*)^{r^*}$ and $j\in (\mathbb{N}^*)^{r^*}$ satisfying $ j_1 < \cdots < j_{r^*} $ with $j_1 \leq N$ and $\abs{\sigma}_1 \leq r$, we have
\[\abs{\Omega_{j,\sigma}(V)} \geq \gamma_{r,N} j_{r^*}^{-2r^*}.\]
\end{prop}

\begin{proof}
Being given $j$ satisfying the above assumptions, we consider the index $k$ given by Corollary \fcolorbox{red}{ white}{\ref{ikea2}}. We aim at estimating  
\[\mathbb{P}(\hspace{0.1cm}\abs{\Omega_{j,\sigma}(V)}< \gamma \quad \text{ and } \quad {\Vert V \Vert}_{\widehat{H}^1}\leq \frac{\gamma_{r}}{2r}N^{-1/6}\hspace{0.1cm})\]
for $\gamma >0$. For this and following (\fcolorbox{red}{ white}{\ref{ikea998}}), we write $V=g_{2k-1}h_{2k-1}(x\sqrt{2})P_{2k-1} + V_{2k-1}$ with $g_{2k-1}$ and $ V_{2k-1}$ independent. Then, we get
\begin{align*}
    \mathbb{P}(\hspace{0.1cm}\abs{\Omega_{j,\sigma}(V)}< &\gamma \quad \text{ and } \quad {\Vert V \Vert}_{\widehat{H}^1}\leq \frac{\gamma_{r}}{2r}N^{-1/6}\hspace{0.1cm})\\
    &= \mathbb{E}\big[\int_{G_{2k-1}\in \mathcal{I}}\mathds{1}_{\abs{\Omega_{j,\sigma} (G_{2k-1}h_{2k-1}(x\sqrt{2})P_{2k-1} + V_{2k-1})}< \gamma} \hspace{0.1cm}f(G_{2k-1})\, \mathrm{d}G_{2k-1}\big]
\end{align*}
where $f(x) = \frac{1}{\sqrt{2\pi}}e^{-x^2/2 }$ denotes the probability density function and the interval
\[\mathcal{I}:= \big\{ G_{2k-1}\in \mathbb{R},\hspace{0.3cm} {\Vert G_{2k-1}h_{2k-1}(\cdot\sqrt{2})P_{2k-1} + V_{2k-1} \Vert}_{\widehat{H}^1}^2\leq \left(\frac{\gamma_{r}}{2r}N^{-1/6}\right)^2 \big\}.\]
Next, for $G_{2k-1}\in \mathcal{I},$ we apply a change of variable $y_{2k-1} = G_{2k-1}P_{2k-1}\in \tilde{\mathcal{I}}$ to get
\begin{align*}
    \mathbb{P}(\hspace{0.1cm}\abs{\Omega_{j,\sigma}(V)}< &\gamma \quad \text{ and } \quad {\Vert V \Vert}_{\widehat{H}^1}\leq \frac{\gamma_{r}}{2r}N^{-1/6}\hspace{0.1cm})\\
    &\leq \frac{P_{2k-1}^{-1}}{\sqrt{2\pi}}\hspace{0.1cm} \mathbb{E}\big[\int_{y_{2k-1}\in \tilde{\mathcal{I}}}  \mathds{1} _{\abs{\Omega_{j,\sigma} (y_{2k-1}h_{2k-1}(x\sqrt{2}) + V_{2k-1})}< \gamma} \, \mathrm{d}y_{2k-1}\big].
\end{align*}
\textcolor{black}{which holds since $0 < f(x) \leq \frac{1}{\sqrt{2\pi}}$.} Now, notice that for $j_1 \leq N,$ Corollary \fcolorbox{red}{ white}{\ref{ikea2}} gives that 
\begin{align}\label{ber10}
\abs{\partial_{y_{2k-1}} \Omega_{j,\sigma}(y_{2k-1}h_{2k-1}(x\sqrt{2}) + V_{2k-1})} &\geq \frac{\gamma_{r}}{2}j_1^{-1/4} \geq \frac{\gamma_{r}}{2}N^{-1/4} .\end{align}
So, since $\tilde{\mathcal{I}}$ is a random interval,
then the map
\[  \Phi: \Bigg \{ \begin{array}{ccc}  \tilde{\mathcal{I}} & \to & \mathcal{J} \\ y_{2k-1} &\mapsto & \Omega_{j,\sigma} (y_{2k-1}h_{2k-1}(x\sqrt{2}) + V_{2k-1})\end{array} \]
 is a diffeomorphism from $\tilde{\mathcal{I}}$ onto its image $\mathcal{J}.$ Moreover, due to the fact that the function $ \mathds{1} _{\abs{\Omega_{j,\sigma} (y_{2k-1}h_{2k-1}(x\sqrt{2}) + V_{2k-1})}< \gamma}$ is integrable on $\tilde{\mathcal{I}}$, we deduce by using the change of variable theorem that the function $ \mathds{1} _{\abs{y_{2k-1}}< \gamma} \abs{\partial_{y_{2k-1}} \Omega_{j,\sigma}(y_{2k-1}h_{2k-1}(x\sqrt{2}) + V_{2k-1})}^{-1}$ is integrable on $\mathcal{J}$  and we have 
 \begin{align}\label{ber11}
\nonumber&\int_{y_{2k-1}\in \tilde{\mathcal{I}}}  \mathds{1} _{\abs{\Omega_{j,\sigma} (y_{2k-1}h_{2k-1}(x\sqrt{2}) + V_{2k-1})}< \gamma} \, \mathrm{d}y_{2k-1}\\
&\hspace{2cm} = \int_{y_{2k-1}\in \mathcal{J}}  \mathds{1} _{\abs{y_{2k-1}}< \gamma} \abs{\partial_{y_{2k-1}} \Omega_{j,\sigma}(y_{2k-1}h_{2k-1}(x\sqrt{2}) + V_{2k-1})}^{-1} \, \mathrm{d}y_{2k-1}.
 \end{align} 
 Thus, making use of (\fcolorbox{red}{ white}{\ref{ber10}}) and (\fcolorbox{red}{ white}{\ref{ber11}}), we obtain the following estimation
  \begin{align*}
  \nonumber&\mathbb{P}(\hspace{0.1cm}\abs{\Omega_{j,\sigma}(V)}< \gamma \quad \text{ and } \quad {\Vert V \Vert}_{\widehat{H}^1}\leq \frac{\gamma_{r}}{2r}N^{-1/6}\hspace{0.1cm})\\
   \nonumber &\leq P_{2k-1}^{-1}\hspace{0.1cm} \mathbb{E}\bigg[\int_{y_{2k-1}\in \mathcal{J}}  \mathds{1} _{\abs{y_{2k-1}}< \gamma} \abs{\partial_{y_{2k-1}} \Omega_{j,\sigma}(y_{2k-1}h_{2k-1}(x\sqrt{2}) + V_{2k-1})}^{-1} \, \mathrm{d}y_{2k-1}\bigg]\\
    \nonumber&\leq 2 P_{2k-1}^{-1}\hspace{0.1cm} \mathbb{E}\bigg[  \gamma_{r}^{-1} N^{1/4}\underbrace{\int_{y_{2k-1}\in \mathcal{J}}  \mathds{1} _{\abs{y_{2k-1}}< \gamma} \, \mathrm{d}y_{2k-1}}_{\leq 2\gamma}\bigg]\\ \label{ikea3}
    &\leq 4\gamma_{r}^{-1}P_{2k-1}^{-1} N^{1/4}\gamma.
 \end{align*}
 Using  (\fcolorbox{red}{ white}{\ref{my1}}) and the fact that $k \lesssim_r N,$ it is possible to control $P_{2k-1}^{-1}$ independently from $j.$
 As a consequence, we get \begin{align*}
  &\mathbb{P}(\hspace{0.1cm}\exists(r^*,\sigma,j), \,\abs{\Omega_{j,\sigma}(V)}< \gamma j_{r^*}^{-2r^*} \quad \text{ and } \quad {\Vert V \Vert}_{\widehat{H}^1}\leq \frac{\gamma_{r}}{2r}N^{-1/6}\hspace{0.1cm})\\
  &\leq \sum_{(r^*,\sigma,j)} \mathbb{P}(\hspace{0.1cm} \abs{\Omega_{j,\sigma}(V)}< \gamma j_{r^*}^{-2r^*} \quad \text{ and } \quad {\Vert V \Vert}_{\widehat{H}^1}\leq \frac{\gamma_{r}}{2r}N^{-1/6}\hspace{0.1cm})\\
  &\lesssim_{r,N} 4\gamma_{r}^{-1} \left(\sum_{(r^*,\sigma,j)}   j_{r^*}^{-2r^*} \right)\gamma. 
  \end{align*}
 The convergence of this last sum is related to the fact that $j_{r^*}$ is the largest index\footnote{Note that the sum with respect to $r^*$ and $\sigma$ is finite.}. So, 
\[\mathbb{P}(\hspace{0.1cm}\exists(r^*,\sigma,j), \,\abs{\Omega_{j,\sigma}(V)}< \gamma j_{r^*}^{-2r^*} \quad \text{ and } \quad {\Vert V \Vert}_{\widehat{H}^1}\leq \frac{\gamma_{r}}{2r}N^{-1/6}\hspace{0.1cm}) \lesssim_{r,N} \gamma \xrightarrow{\text{ as }\gamma \hspace{0.05cm} \rightarrow \hspace{0.05cm} 0} 0.\]
It is natural to conclude that since the probability vanishes, almost surely there exists $\gamma>0$ depending on $r, N$ and $V$ such that for all $(r^*,\sigma,j)$ satisfying the given assumptions, we have
\[\abs{\Omega_{j,\sigma}(V)} \geq \gamma j_{r^*}^{-2r^*}. \qedhere\]
\end{proof}

Now, we have reached the proof of Theorem \fcolorbox{red}{ white}{\ref{ber51}} which is the main result of this section. More precisely, we obtain the strong $N,r$ non-resonance condition of the frequencies ${(w_j)}_{j\in \mathbb{N}^*}$ of the quantum harmonic oscillator with a perturbation.

\textbf{Proof of Theorem \fcolorbox{red}{ white}{\ref{ber51}}}
To start, we can directly see that the localization hypothesis (\fcolorbox{red}{ white}{\ref{ez3}}) on the spectrum is obtained in Lemma \fcolorbox{red}{ white}{\ref{N7}}. So, the frequencies are close to integer values, and we have that there exists a constant $C>0$ such that
\[\sigma \hspace{0.1cm}(T+V) \subset \bigcup_{j\geq 1} \mathcal{Z}_j \quad \text{with} \quad \mathcal{Z}_j:= \big[2j-1-Cj^{-1/2}, 2j-1+Cj^{-1/2}\big].\]
We just proved (\fcolorbox{red}{ white}{\ref{eq7}}) in Proposition \fcolorbox{red}{ white}{\ref{ikea5}} where we obtained a control of the small divisors by the smallest index involved. Finally, our result is a direct consequence of Proposition \fcolorbox{red}{ white}{\ref{pr55}}.

\begin{RQ}
The key point related to our model is that the Birkhoff normal form procedure described in Section \fcolorbox{red}{ white}{\ref{z3}} involves small divisors defined by
\begin{align}
\label{ret2}\Omega_{j,\ell}(V) &= w_{j_1}+\cdots + w_{j_{r}}-w_{\ell_{1}}-\cdots - w_{\ell_{r}}
\end{align} 
where \textcolor{black}{the sequence ${(w_j)}_{j\geq 1}$ stands for the sequence ${(\Lambda_j)}_{j\geq 1}$ and denotes the frequencies of the perturbed harmonic oscillator.} Furthermore, a same term may appear both with a positive and a negative sign. Therefore, it  is sufficient to define the minimum index as:
\begin{align}\label{kappa} \kappa(j,\ell)&= \min \{ \hspace{0.3cm} s_i := j_i, \ell_i ,\hspace{0.2cm} 1\leq i \leq r \quad\text{and}\quad  \sum_{n =1 }^{r} \left(\mathds{1}_{j_n=s_i} - \mathds{1}_{\ell_n=s_i} \right) \neq 0 \hspace{0.3cm}\}\cup \{\infty\}.\end{align} 
 As a result, we establish a generalisation to Definition \fcolorbox{red}{ white}{\ref{dep1}} and a suitable formalism for the Birkhoff normal form process by providing a uniform bound for the small divisors $\Omega_{j,\ell}(V)$ given in (\fcolorbox{red}{ white}{\ref{ret2}}).
 \end{RQ}
   
\begin{prop}\label{po1}
Let $V$ be given in (\fcolorbox{red}{ white}{\ref{ikea998}}). For all $r,$ $N\geq 1,$ provided that ${\Vert V \Vert}_{\widehat{H}^1}\lesssim_r  N^{-1/6},$ there exists $\beta_{r,N}>0$ such that for all $j,\ell \in (\mathbb{N^*})^r$, if $\kappa(j,\ell) \leq N,$ we either have
\begin{align*} \abs{\Omega_{j,\ell}(V)} &\geq \beta_{r,N} \end{align*} or the small divisor is trivial and we write, in this case, $\kappa(j,\ell)=\infty.$
\end{prop} 

\begin{RQ}
 We notice that this control rather than the control of the small divisors by the third largest index (known as the standard non-resonance condition whose explicit definition is found for instance in \cite{zbMATH05117851}) will allow us to remove much more terms when solving the cohomological equations in the Birkhoff normal form process.
\end{RQ}
   
\section{Hamiltonian formalism}\label{z2}
We are going to introduce here a Hamiltonian class which plays an important role in classifying the Hamiltonian polynomials arising in the proof of the Birkhoff normal form theorem. Roughly speaking, the Hamiltonian polynomials in the normal form process are controlled by the $\mathscr{H}$-norm whereas the solutions to the cohomological equation are controlled by the $\mathscr{C}$-norm (see Definition \fcolorbox{red}{ white}{\ref{ikea11}}).

\subsection{Functional setting}\label{W1}
 We fix $M>1$, and we note that we are working in finite dimension. In other words, $h^{1/2}(\llbracket 1,M \rrbracket) \equiv \mathbb{C}^{\llbracket 1,M \rrbracket}$ is a finite dimensional vector space.
\begin{defi}(Natural Scalar Product)
We equip $\ell^2(\llbracket 1,M \rrbracket)$ with its natural real scalar product
\[{(u,v)} _{\ell^2} := \sum_{k\in \llbracket 1,M \rrbracket} \Re \hspace{0.05cm}\overline{u_k}v_k = \sum_{k\in \llbracket 1,M \rrbracket} \left( \Re \hspace{0.05cm}u_k \Re \hspace{0.05cm}v_k + \Im\hspace{0.05cm}u_k \Im\hspace{0.05cm}v_k\right) \in\mathbb{R}.\]
\end{defi}


\begin{defi}(Poisson Bracket)
Let $H,K {\;}: \,  \mathbb{C}^{\llbracket 1,M \rrbracket} \to \mathbb{R}$ be two smooth functions. Then the Poisson bracket of $H$ and $K$ is defined by:
\[ \{H,K\}(u):={( i\nabla H(u),\nabla K(u))}_{\ell^2}\] 
where $\nabla H(u)=2{(\partial_{\overline{u_k}}H(u))}_{k}.$
\end{defi}

\begin{lem}\label{ep1}
We have the following identity
\[\{H,K\}(u) = 2i\sum_{k\in \llbracket 1,M \rrbracket} \left(\partial_{\overline{u_k}}H(u) \partial_{u_k}K(u) - \partial_{u_k}H(u) \partial_{\overline{u_k}}K(u)\right).\]
\end{lem}

\begin{proof}
To see this, we write using the definition
\[\{H,K\}(u)={(i\nabla H(u),\nabla K(u))}_{\ell^2}
= 4\sum_{k\in \llbracket 1,M \rrbracket} \Re \hspace{0.05cm}i\partial_{\overline{u_k}}H(u) \overline{\partial_{\overline{u_k}}K(u)}.\]
By simple calculations, one can prove that
\[4\Re \hspace{0.05cm} i\partial_{\overline{u_k}}H(u) \overline{\partial_{\overline{u_k}}K(u)} = 2i \left(\partial_{\overline{u_k}}H(u) \partial_{u_k}K(u) - \partial_{u_k}H(u) \partial_{\overline{u_k}}K(u)\right). \qedhere\]
\end{proof}

\begin{defi}(Symplectic Map)\label{defff4}
Consider an open set $\mathcal{C}$ of $\mathbb{C}^{\llbracket 1,M \rrbracket}$ and a $C^1$ map $\tau {\;}: \,  \mathcal{C} \to \mathbb{C}^{\llbracket 1,M \rrbracket}.$ We say that $\tau$ is a symplectic map if
 \[ \forall u\in \mathcal{C}, \forall v,w \in \mathbb{C}^{\llbracket 1,M \rrbracket}, \hspace{0.2cm} {( iv,w)} _{\ell^2} = {( i \mathrm{d}\tau (u)(v), \mathrm{d}\tau (u)(w))} _{\ell^2}.\]
\end{defi}

\subsection{Class of Hamiltonian functions}\label{2001}
\begin{defi}(Class $\mathscr{H}^{2r}_M)$\label{defi1}
Being given $M > 1$ and $r\geq 1$, we denote by $\mathscr{H}^{2r}_M$ the set of real valued homogeneous polynomials of degree $2r$ defined on $\mathbb{C}^{\llbracket 1,M \rrbracket}$. Consequently, these Hamiltonians are uniquely written as
\[ H(u)= \sum_{\substack{j,\ell \in \llbracket 1,M \rrbracket^{r}}}H_{j,\ell}\hspace{0.05cm} u_{j_1}\cdots u_{j_r} \overline{u_{\ell_{1}}}\cdots \overline{u_{\ell_{r}}} \] 
where $(H_{j,\ell})_{(j,\ell)\in \llbracket 1,M \rrbracket^{r} \times \llbracket 1,M \rrbracket^{r}} $ is a sequence of complex numbers satisfying:
\begin{itemize}
 \item the \textit{reality condition} $$H_{j,\ell}=\overline{H_{\ell,j}}$$
    \item the \textit{symmetry condition} $$\forall (\phi,\psi) \in \mathscr{S}_{r}\times \mathscr{S}_{r}, \hspace{0.3cm} H_{j_1,\cdots , j_{r},\ell_{1},\cdots \ell_r} = H_{j_{\phi_1},\cdots , j_{\phi_{r}},\ell_{\psi_{1}},\cdots \ell_{\psi_r}}.$$
   
\end{itemize}
\end{defi}

We endow this space of polynomials with the two following norms ${\Vert\cdot\Vert}_{\mathscr{H}}$ and ${\Vert \cdot\Vert}_{\mathscr{C}}.$ 
\begin{defi}(Norms ${\Vert \cdot \Vert}_{\mathscr{H}}$ and ${\Vert \cdot \Vert}_{\mathscr{C}}$) \label{ikea11}
Let $M> 1,$ $r\geq 1$ and $H,\chi \in \mathscr{H}^{2r}_M.$ We introduce the norms
\[{\Vert H\Vert}_{\mathscr{H}}:= \sup_{j,\ell\in \llbracket 1,M \rrbracket^{r}}\abs{H_{j,\ell}} \] and \[{\Vert \chi\Vert}_{\mathscr{C}}:= \sup_{j,\ell \in \llbracket 1,M \rrbracket^{r}}\abs{\chi_{j,\ell}} \langle j_1+\cdots+j_{r}-\ell_{1}-\cdots -\ell_{r} \rangle. \] 
\end{defi}

We will show two essential lemmas needed to establish the continuity estimates enjoyed by the Hamiltonians. The first lemma states the following:

\begin{lem}\label{eq90}
For a Hamiltonian $H \in \mathscr{H}_M^{2r}$ and $u^{(1)},\cdots, u^{(2r)} \in \mathbb{C}^{\llbracket 1,M \rrbracket},$ we have 
\[\sum_{j,\ell\in \llbracket 1,M \rrbracket^{r}} \abs{H_{j,\ell} u_{j_1}^{(1)}\cdots u_{j_{r}}^{(r)} \overline{u_{\ell_{1}}^{(r+1)}}\cdots \overline{u_{\ell_{r}}^{(2r)}}} \leq (\log M)^{r}{\Vert H\Vert}_{\mathscr{H}} \prod_{i=1}^{2r}{\Vert u^{(i)}\Vert}_{h^{1/2}}.\]
\end{lem} 

\begin{proof}
Let $H \in \mathscr{H}^{2r}_M$ and $ u^{(1)},\cdots, u^{(2r)} \in \mathbb{C}^{\llbracket 1,M \rrbracket}.$ We then write
\begin{align*}
    \sum_{j,\ell \in \llbracket 1,M \rrbracket^{r}} \abs{H_{j,\ell} u_{j_1}^{(1)}\cdots u_{j_{r}}^{(r)} \overline{u_{\ell_{1}}^{(r+1)}}\cdots \overline{u_{\ell_{r}}^{(2r)}}}
    &\leq \sum_{j,\ell \in \llbracket 1,M \rrbracket^{r}}{\Vert H\Vert}_{\mathscr{H}} \abs{ u_{j_1}^{(1)}}\cdots \abs{u_{j_{r}}^{(r)}}\abs{\overline{u_{\ell_{1}}^{(r+1)}}}\cdots \abs{\overline{u_{\ell_{r}}^{(2r)}}} \\
       &\leq{\Vert H\Vert}_{\mathscr{H}}\prod_{i=1}^{2r}\left(\sum_{k\in \llbracket 1,M \rrbracket} \langle k \rangle^{1/2}\abs{ u_{k}^{(i)}}\frac{1}{\langle k \rangle^{1/2}}\right).
\end{align*}
By Cauchy--Schwarz inequality and the fact that $\sum\limits_{k \in \llbracket 1,M \rrbracket}\frac{1}{\langle k \rangle} \lesssim \log{M}$, 
we obtain  
\begin{align*}
    \sum_{j,\ell \in \llbracket 1,M \rrbracket^{r}} \abs{H_{j,\ell} u_{j_1}^{(1)}\cdots u_{j_{r}}^{(r)} \overline{u_{\ell_{1}}^{(r+1)}}\cdots \overline{u_{\ell_{r}}^{(2r)}}}
    &\leq {\Vert H\Vert}_{\mathscr{H}}\left(\sum_{k\in \llbracket 1,M \rrbracket} \frac{1}{\langle k \rangle}\right)^{r}\prod_{i=1}^{2r}\left(\sum_{k\in \llbracket 1,M \rrbracket} \langle k \rangle\abs{ u_{k}^{(i)}}^2\right)^{1/2}\\
     &\lesssim (\log M)^{r}{\Vert H\Vert}_{\mathscr{H}}\prod_{i=1}^{2r}{||u^{(i)}||}_{h^{1/2}}.
     \end{align*}
\end{proof}


The second lemma seems a bit more complicated and writes as follows:
\begin{lem}\label{eqq1}
For all $u^{(1)},\cdots, u^{(2r)} \in \mathbb{C}^{\llbracket 1,M \rrbracket},$ we have 
\begin{align*}
\sum_{j,\ell \in \llbracket 1,M \rrbracket^{r}}\frac{1}{\langle j_1+\cdots+j_{r}-\ell_{1}-\cdots- \ell_{r}\rangle} \prod_{i=1}^{r}\abs{u_{j_i}^ {(i)}} \abs{u_{\ell_i}^ {(r+i)}}
& \lesssim_r (\log M)^{r} {\Vert u^{(2r)}\Vert}_{h^{-1/2}} \prod_{i=1}^{2r-1}{\Vert u^{(i)}\Vert}_{h^{1/2}}.
\end{align*}
\end{lem}

\begin{proof}
Denote $v_k:= \langle k \rangle ^{-1/2} u_k^{(2r)}$ and $j_0=-(j_1+\cdots+j_{r}-\ell_{1}-\cdots- \ell_{r}).$ We can easily notice that ${\Vert v \Vert}_{\ell^2}={\Vert u^{(2r)} \Vert}_{h^{-1/2}}$, and we write
\begin{align*}
  &\sum_{j,\ell \in \llbracket 1,M \rrbracket^{r}}\frac{1}{\langle j_1+\cdots+j_{r}-\ell_{1}-\cdots - \ell_{r}\rangle}  \prod_{i=1}^{r}\abs{u_{j_i}^ {(i)}} \abs{u_{\ell_i}^ {(r+i)}}\\
  &\hspace{7cm}=\sum_{j,\ell \in \llbracket 1,M \rrbracket^{r}}\frac{1}{\langle j_0 \rangle} \langle \ell_r \rangle ^{1/2} \abs{v_{\ell_r}}\prod_{i=1}^{r}\abs{u_{j_i}^ {(i)}} \prod_{i=1}^{r-1}\abs{u_{\ell_i}^ {(r+i)}}.
\end{align*}
Now, in order to get rid of the term $\langle \ell_{r} \rangle^{1/2},$ we use Jensen's formula to obtain 
\[\langle \ell_{r} \rangle^{1/2} = \langle j_0+\cdots+j_{r}-\ell_{1}-\cdots- \ell_{r-1}\rangle^{1/2}\leq (\langle j_0 \rangle+\cdots+\langle \ell_{r-1} \rangle)^{1/2}\leq \sum_{n=0}^{r}\langle j_{n} \rangle^{1/2} + \sum_{n=1}^{r-1}\langle \ell_{n} \rangle^{1/2}.\]
Consequently, we get 
\begin{align*}
    &\sum_{j,\ell \in \llbracket 1,M \rrbracket^{r}}\frac{1}{\langle j_1+\cdots+j_{r}-\ell_{1}-\cdots - \ell_{r}\rangle}  \prod_{i=1}^{r}\abs{u_{j_i}^ {(i)}} \abs{u_{\ell_i}^ {(r+i)}}\\
    &\hspace{1.5cm}\leq \sum_{j,\ell \in \llbracket 1,M \rrbracket^{r}}\frac{1}{\langle j_0 \rangle} \left(  \sum_{n=0}^{r}\langle j_{n} \rangle^{1/2} + \sum_{n=1}^{r-1}\langle \ell_{n} \rangle^{1/2} \right) \abs{v_{\ell_r}}\prod_{i=1}^{r}\abs{u_{j_i}^ {(i)}} \prod_{i=1}^{r-1}\abs{u_{\ell_i}^ {(r+i)}}\\
     & \hspace{1.5cm} \leq\sum_{j,\ell \in \llbracket 1,M \rrbracket^{r}} \Bigg[ \frac{1}{\langle j_0 \rangle^{1/2}} + \frac{1}{\langle j_0 \rangle} \left(  \sum_{n=1}^{r}\langle j_{n} \rangle^{1/2} + \sum_{n=1}^{r-1}\langle \ell_{n} \rangle^{1/2} \right) \Bigg] \abs{v_{\ell_{r}}}\prod_{i=1}^{r}\abs{u_{j_i}^ {(i)}} \prod_{i=1}^{r-1}\abs{u_{\ell_i}^ {(r+i)}}.
    \end{align*}
Notice that \[\left(\sum_{n=1}^{r}\langle j_{n} \rangle^{1/2}\right) \prod_{i=1}^{r}\abs{u_{j_i}^ {(i)}} \prod_{i=1}^{r-1}\abs{u_{\ell_i}^ {(r+i)}} = \left(\sum_{n=1}^{r}\langle j_{n} \rangle^{1/2}\abs{u_{j_n}^{(n)}} \prod_{\substack{i=1 \\ i\neq n}}^{r} \abs{u_{j_i}^ {(i)}}\right)  \prod_{i=1}^{r-1}\abs{u_{\ell_i}^ {(r+i)}} \]
and similarly that 
  \[\left(\sum_{n=1}^{r-1}\langle \ell_{n} \rangle^{1/2}\right) \prod_{i=1}^{r}\abs{u_{j_i}^ {(i)}} \prod_{i=1}^{r-1}\abs{u_{\ell_i}^ {(r+i)}} = \left(\sum_{n=1}^{r-1}\langle \ell_{n} \rangle^{1/2}\abs{u_{\ell_n}^{(r+n)}} \prod_{\substack{i=1 \\ i\neq n}}^{r-1} \abs{u_{\ell_i}^ {(r+i)}}\right)  \prod_{i=1}^{r}\abs{u_{j_i}^ {(i)}}.\]
  
Thus, we obtain 
\begin{align}
  \nonumber&\sum_{j,\ell \in \llbracket 1,M \rrbracket^{r}}\frac{1}{\langle j_1+\cdots+j_{r}-\ell_{1}-\cdots - \ell_{r}\rangle}  \prod_{i=1}^{r}\abs{u_{j_i}^ {(i)}} \abs{u_{\ell_i}^ {(r+i)}}\\
   &\hspace{2cm}\leq \nonumber\sum_{j,\ell \in \llbracket 1,M \rrbracket^{r}} \frac{1}{\langle j_0 \rangle^{1/2}} \abs{v_{\ell_{r}}}\prod_{i=1}^{r}\abs{u_{j_i}^ {(i)}} \prod_{i=1}^{r-1}\abs{u_{\ell_i}^ {(r+i)}} \\
  \label{eq1}  &\hspace{3cm}+ \sum_{j,\ell \in \llbracket 1,M \rrbracket^{r}}  \abs{v_{\ell_{r}}} \frac{1}{\langle j_0 \rangle} \left(\sum_{n=1}^{r}\langle j_{n} \rangle^{1/2}\abs{u_{j_n}^{(n)}} \prod_{\substack{i=1 \\ i\neq n}}^{r} \abs{u_{j_i}^ {(i)}}\right)  \prod_{i=1}^{r-1}\abs{u_{\ell_i}^ {(r+i)}}\\
   &\hspace{3cm}+ \nonumber \sum_{j,\ell \in \llbracket 1,M \rrbracket^{r}}  \abs{v_{\ell_{r}}} \frac{1}{\langle j_0 \rangle} \left(\sum_{n=1}^{r-1}\langle \ell_{n} \rangle^{1/2}\abs{u_{\ell_n}^{(r+n)}} \prod_{\substack{i=1 \\ i\neq n}}^{r-1} \abs{u_{\ell_i}^ {(r+i)}}\right)  \prod_{i=1}^{r}\abs{u_{j_i}^ {(i)}}.
\end{align}

\makeatletter
\DeclareRobustCommand\bigop[2][1]{%
  \mathop{\vphantom{\sum}\mathpalette\bigop@{{#1}{#2}}}\slimits@
}
\newcommand{\bigop@}[2]{\bigop@@#1#2}
\newcommand{\bigop@@}[3]{%
  \vcenter{%
    \sbox\z@{$#1\sum$}%
    \hbox{\resizebox{\ifx#1\displaystyle#2\fi\dimexpr\ht\z@+\dp\z@}{!}{$\m@th#3$}}%
  }%
}
\makeatother
\newcommand{\bigstarr}{\DOTSB\bigop{\star}}

It turns out that the sum we aim at estimating writes as a convolution product, and the needed result is just a consequence of Young's convolution inequality $$\ell^2 * \ell^2*\ell^1* \cdots *\ell^1 \hookrightarrow \ell^{\infty}.$$ Therefore (\fcolorbox{red}{ white}{\ref{eq1}}) can be expressed as
\begin{align*} 
&\left[\abs{v}*\langle \cdot \rangle^{-1/2}*\abs{u^{(1)}}*\cdots *\abs{u^{(2r-1)}}\right]_0\\
& \hspace{2.5cm}+\left[\sum_{n=1}^{r}\abs{v}*\langle \cdot \rangle^{1/2}\abs{u^{(n)}}*\langle \cdot \rangle^{-1}*\Bigg(\mathop{\scalebox{2}{\raisebox{-0.2ex}{$\ast$}}}_{\substack{i=1\\ i\neq n}}^{r}\abs{u_{j_i}^ {(i)}}\Bigg) *\Bigg(\mathop{\scalebox{2}{\raisebox{-0.2ex}{$\ast$}}}_{\substack{i=1}}^{r-1}\abs{u_{\ell_i}^ {(r+i)}}\Bigg)\right]_0\\
&\hspace{2.5cm}+\left[\sum_{n=1}^{r-1}\abs{v}*\langle \cdot \rangle^{1/2}\abs{u^{(r+n)}}*\langle \cdot \rangle^{-1}*\Bigg(\mathop{\scalebox{2}{\raisebox{-0.2ex}{$\ast$}}}_{\substack{i=1 \\ i\neq n}}^{r-1}\abs{u_{\ell_i}^ {(r+i)}}\Bigg) *\Bigg(\mathop{\scalebox{2}{\raisebox{-0.2ex}{$\ast$}}}_{\substack{i=1}}^{r}\abs{u_{j_i}^ {(i)}}\Bigg) \right]_0\\
&\leq \label{eq4} {\Vert v \Vert}_{\ell ^2}{\Vert \langle \cdot \rangle^{-1/2} \Vert}_{\ell ^2}\prod_{i=1}^{2r-1}{\Vert u^{(i)} \Vert}_{\ell ^1} +\sum_ {n=1}^{r}{\Vert v \Vert}_{\ell^2}{\Vert\langle \cdot \rangle^{1/2}u^{(n)} * \langle \cdot  \rangle^{-1}\Vert}_{\ell^2}\prod_{\substack{i=1 \\ i\neq n}}^{r}{\Vert u^ {(i)}\Vert}_{\ell^1} \prod_{\substack{i=1}}^{r-1}{\Vert u^ {(r+i)}\Vert}_{\ell^1}\\
&\hspace{5.2cm}+\sum_ {n=1}^{r-1}{\Vert v \Vert}_{\ell^2}{\Vert\langle \cdot \rangle^{1/2}u^{(r+n)} * \langle \cdot  \rangle^{-1}\Vert}_{\ell^2}\prod_{\substack{i=1 \\ i\neq n}}^{r-1}{\Vert u^ {(r+i)}\Vert}_{\ell^1} \prod_{\substack{i=1}}^{r}{\Vert u^ {(i)}\Vert}_{\ell^1}.
\end{align*}

We are left with proving the following estimates for the $\ell^1$-norm and the $\ell^2$-norm.\\
\underline{Estimate of ${\Vert u^{(i)} \Vert}_{\ell ^1}$}: Using Cauchy--Schwarz inequality, we get
\begin{align*}{\Vert u^{(i)} \Vert}_{\ell ^1}= {\bigg\Vert u^{(i)}\frac{\langle \cdot \rangle^ {1/2}}{\langle \cdot \rangle^{1/2}} \bigg\Vert}_{\ell ^1}\leq {\Vert u^{(i)} \Vert} _{h ^{1/2}} {\Vert\langle \cdot \rangle^{-1/2}\Vert}_{\ell ^2} \leq (\log M)^{1/2}{\Vert u^{(i)} \Vert} _{h ^{1/2}}. \end{align*}

\underline{Estimate of ${\Vert\langle\cdot\rangle^{1/2}u^{(n)}* \langle\cdot \rangle^{-1}\Vert}_{\ell^2}$}: Apply Young's convolution inequality $\ell^2 * \ell^1 \hookrightarrow \ell^{2}$ to get
\[{\Vert\langle\cdot\rangle^{1/2}u^{(n)}* \langle\cdot \rangle^{-1}\Vert}_ {\ell^2}\leq {\Vert\langle\cdot\rangle^{1/2}u^{(n)} \Vert}_{\ell^2}{\Vert \langle\cdot \rangle^{-1}\Vert}_{\ell^1}\leq (\log M){\Vert u^{(n)} \Vert} _{h ^{1/2}}. \qedhere\]
\end{proof}

Now we turn to the estimate on the gradient provided by the $\mathscr{H}$-norm and an even better estimate provided by the $\mathscr{C}$-norm.
\begin{prop}\label{PLKI}
Let $M\geq 2$, $r \geq  1$. For all $H\in \mathscr{H}^{2r}_M,$ the gradient of $H$ is a smooth function 
  enjoying the bound 
  \[\forall u \in \mathbb{C}^{\llbracket 1,M \rrbracket}, \hspace{0.3cm} {\Vert\nabla H(u)\Vert}_{h^{-1/2}}\lesssim_r (\log M)^{r} {\Vert H \Vert}_{\mathscr{H}} {\Vert u \Vert}_{h^{1/2}}^{2r-1}.\]
\end{prop}

\begin{proof}
The proof is obtained by duality. We fix $v\in \mathbb{C}^{\llbracket 1,M \rrbracket}$, and we write
\[ {\Vert \nabla H(u)\Vert}_{h^{-1/2}} = \sup_{{\Vert v \Vert}_{h^{1/2}}\leq 1} \abs{{(\nabla H(u),v)}_{\ell^2}}.\]
Notice that since $${(\nabla H(u),v)}_{\ell^2} = 2r\sum_{j,\ell \in \llbracket 1,M \rrbracket^{r}} \Re \hspace{0.03cm}\left[ H_{j,\ell} u_{j_1}\cdots u_{j_{r}} \overline{u_{\ell_{1}}} \cdots \overline{v_{l_{r}}}\right],$$ then the needed result is a direct corollary of Lemma \fcolorbox{red}{ white}{\ref{eq90}}.
\end{proof}

\begin{prop}\label{grad}
 Let $M\geq 2$, $r \geq  1$. For all $\chi\in \mathscr{H}^{2r}_M$ and all $u \in \mathbb{C}^{\llbracket 1,M \rrbracket},$ the gradient of $\chi$ enjoys the bounds 
  \[ {\Vert\nabla \chi(u)\Vert}_{h^{1/2}}\lesssim_r (\log M)^{r} {\Vert \chi \Vert}_{\mathscr{C}} {\Vert u \Vert}_{h^{1/2}}^{2r-1}\] and 
  \[ {\Vert \mathrm{d}\nabla \chi(u)\Vert}_{\mathcal{L}(h^{1/2})}\lesssim_r (\log M)^{r} {\Vert \chi \Vert}_{\mathscr{C}} {\Vert u \Vert}_{h^{1/2}}^{2r-2}. \]
\end{prop}

\begin{proof}
It is similar to the proof of Proposition \fcolorbox{red}{ white}{\ref{PLKI}}, except that we use Lemma \fcolorbox{red}{ white}{\ref{eqq1}} instead of Lemma \fcolorbox{red}{ white}{\ref{eq90}}.
\end{proof}

As a consequence, the second estimate of Proposition \fcolorbox{red}{ white}{\ref{grad}} can be written in the negative Sobolev space $h^{-1/2}$ as follows:
\begin{cor}
 Let $M\geq 2$, $r \geq  1$. For all $\chi\in \mathscr{H}^{2r}_M$ and all $u \in \mathbb{C}^{\llbracket 1,M \rrbracket},$ we have
  \[ {\Vert \mathrm{d}\nabla \chi(u)\Vert}_{\mathcal{L}(h^{-1/2})}\lesssim_r (\log M)^{r} {\Vert \chi \Vert}_{\mathscr{C}} {\Vert u \Vert}_{h^{1/2}}^{2r-2}. \]
\end{cor}

\begin{proof}
The proof uses a standard duality argument and is found in \cite{bernier:hal-03334431} Corollary 4.7.
\end{proof}

Now, we introduce the flow generated by a Hamiltonian belonging to $\mathscr{H}_M^{2r}.$
\begin{lem}\label{imp1}
Let $M\geq 2$, $r \geq 2$ and $\chi \in \mathscr{H}_M^{2r}$. Then there exists 
\[\varepsilon_1 = \left(K(\log M)^{(2r-1)/2}{\Vert\chi\Vert}_{\mathscr{C}}\right)^{-1/{(2r-2)}}\]
where $K$ depends on $r$, and there exists a smooth map 
\[ \phi_{\chi}: \Bigg \{ \begin{array}{ccc}  [-1,1]\times B_{\mathbb{C}^{\llbracket 1,M \rrbracket}}(0,\varepsilon_1) & \to &  \mathbb{C}^{\llbracket 1,M \rrbracket} \\ (t,u) &\mapsto & \phi_{\chi}^t (u) \end{array} \]
solving the equation $-i\partial_t \phi_{\chi} = (\nabla\chi)\circ \phi_{\chi}$ and satisfying for all $t\in [-1,1]$ the following:
\begin{itemize} 
    \item [1.]close to the identity: $ \forall u\in B_{\mathbb{C}^{\llbracket 1,M \rrbracket}}(0,\varepsilon_1), \hspace{0.1cm} {\Vert \phi_{\chi}^t (u) - u\Vert}_{h^{1/2}} \leq \left(\frac{{\Vert u \Vert}_{h^{1/2}}}{\varepsilon_1}\right)^{2r-2}{\Vert u \Vert}_{h^{1/2}},$
    \item [2.] invertible: $ {\Vert \phi_{\chi}^{-t}( u)\Vert}_{h^{1/2}} < \varepsilon_1 \implies \phi_{\chi}^{t}\circ \phi_{\chi}^{-t}( u) = u,$
    \item [3.]symplectic: recall Definition \fcolorbox{red}{ white}{\ref{defff4}}.
\end{itemize}
Moreover, its differential is a continuous map and enjoys the bound: \[ \forall u\in B_{\mathbb{C}^{\llbracket 1,M \rrbracket}}(0,\varepsilon_1), \forall \sigma \in \{-1,1\}, \hspace{0.1cm} {\Vert \mathrm{d}\phi_{\chi}^t (u) \Vert}_{\mathscr{L}(h^{\sigma/2})} \leq 2.\]
\end{lem}

\begin{proof}
We refer the reader to the proof of Proposition 4.8 in \cite{bernier:hal-03334431}.
\end{proof}

We shall prove after this that the Hamiltonians are stable by the \textcolor{black}{P}oisson brackets.
\begin{prop}\label{1209}
Let $H \in \mathscr{H}^{2r}_M$ and $\chi \in \mathscr{H}^{2r'}_M$ with $r,r' \geq 1.$ Then, there exists a Hamiltonian $N\in \mathscr{H}^{2r+2r'-2}_M$ such that 
$$\forall u \in \mathbb{C}^{\llbracket 1,M \rrbracket},\hspace{0.3cm} \{H,\chi\}(u) = N(u)$$ and $${\Vert\{H,\chi\}\Vert}_{\mathscr{H}} \lesssim_{\textcolor{black}{r,r'}} \log M{\Vert H\Vert}_{\mathscr{H}}{\Vert\chi\Vert}_{\mathscr{C}}.$$
\end{prop}

\begin{proof}
Let $u\in \mathbb{C}^{\llbracket 1,M \rrbracket}.$ We express the Hamiltonians as
\begin{align*}
H(u)&=\sum_{\substack{j,\ell \in \llbracket 1,M \rrbracket^{r}}}H_{j,\ell} u_{j_1}\cdots u_{j_{r}} \overline{u_{\ell_{1}}}\cdots \overline{u_{\ell_{r}}} \quad \text{and}\quad \chi(u)=\sum_{\substack{j',\ell' \in \llbracket 1,M \rrbracket^{r'}}}\chi_{j',\ell'} u_{j'_1}\cdots u_{j'_{r'}} \overline{u_{\ell'_{1}}}\cdots \overline{u_{\ell'_{r'}}}.\end{align*}
By Lemma \fcolorbox{red}{ white}{\ref{ep1}}, we have 
\[\{H,\chi\}(u) = 2i\sum_{k\in \llbracket 1,M \rrbracket} \partial_{\overline{u_k}}H(u) \partial_{u_k}\chi(u) - \partial_{u_k}H(u) \partial_{\overline{u _k}}\chi(u).\] 
Using the symmetry condition satisfied by the coefficients of $H$ and $\chi$, we get
\begin{align*}
    \partial_{\overline{u_k}}H(u) \partial_{u_k}\chi(u)  &= rr' \sum_{\substack{j \in \llbracket 1,M \rrbracket^{r} \\ \ell \in \llbracket 1,M \rrbracket^{r-1} \\ j'\in \llbracket 1,M \rrbracket^{r'-1} \\ \ell' \in \llbracket 1,M \rrbracket^{r'}}}H_{j,\ell,k} u_{j_1}\cdots u_{j_{r}} \overline{u_{\ell_{1}}}\cdots \overline{u_{\ell_{r-1}}} \chi_{j',k,\ell'} u_{j'_1}\cdots u_{j'_{r'-1}} \overline{u_{\ell'_{1}}}\cdots \overline{u_{\ell'_{r'}}}.
\end{align*}
Now, we set $j'':=(j,j')$, $\ell'':=(\ell,\ell')$ and $r'':=r+r'-1$. After reindexing, we can see that
\begin{align*}
    \{H,\chi\}(u) &= 2i\sum_{k\in \llbracket 1,M \rrbracket}\Bigg[rr' \sum_{\substack{j \in \llbracket 1,M \rrbracket^{r} \\ \ell \in \llbracket 1,M \rrbracket^{r-1} \\ j'\in \llbracket 1,M \rrbracket^{r'-1} \\ \ell' \in \llbracket 1,M \rrbracket^{r'}}}H_{j,\ell,k} u_{j_1}\cdots u_{j_{r}} \overline{u_{\ell_{1}}}\cdots \overline{u_{\ell_{r-1}}} \chi_{j',k,\ell'} u_{j'_1}\cdots u_{j'_{r'-1}} \overline{u_{\ell'_{1}}}\cdots \overline{u_{\ell'_{r'}}}\\
    &\hspace{1cm}- rr' \sum_{\substack{j \in \llbracket 1,M \rrbracket^{r-1} \\ \ell \in \llbracket 1,M \rrbracket^{r} \\ j'\in \llbracket 1,M \rrbracket^{r'} \\ \ell' \in \llbracket 1,M \rrbracket^{r'-1}}}H_{j,k,\ell} u_{j_1}\cdots u_{j_{r-1}} \overline{u_{\ell_{1}}}\cdots \overline{u_{\ell_{r}}} \chi_{j',\ell',k} u_{j'_1}\cdots u_{j'_{r'}} \overline{u_{\ell'_{1}}}\cdots \overline{u_{\ell'_{r'-1}}}\Bigg]\\
    &= \underbrace{\sum_{\substack{j'',\ell'' \in \llbracket 1,M \rrbracket^{r''}}}\underbrace{\left(2irr'\sum\limits_{k\in \llbracket 1,M \rrbracket}H_{j,\ell,k} \chi_{j',k,\ell'} - H_{j,k,\ell} \chi_{j',\ell',k}\right)}_{N_{j'',\ell''}} u_{j_1''}\cdots u_{j_{r''}''} \overline{u_{\ell_{1}''}}\cdots \overline{u_{\ell_{r''}''}}}_{N(u)}.
\end{align*}

Note that we can interchange the order of summation since we are dealing with finite sums.
 Moreover, we can obviously see that $N(u)$ defines a homogeneous polynomial of degree $2r''$ (i.e. $N \in \mathscr{H}_M^{2r''}$ \textcolor{black}{where both the symmetry and reality conditions of $N_{j'',\ell''}$ are a direct consequence of those satisfied by $H_{j,\ell}$ and $\chi_{j',\ell'}$}). On the other hand, we need to verify the upper bound on the $\mathscr{H}-$norm. For this, we write
\begin{align*}
    \sum_{k\in \llbracket 1,M \rrbracket}\abs{H_{j,k,\ell} \chi_{j',\ell',k}}&\leq \sum_{k\in \llbracket 1,M \rrbracket} {\Vert H \Vert}_{\mathscr{H}}{\Vert \chi \Vert}_{\mathscr{C}} \frac{1}{\langle j_1'+\cdots+j_{r'}'-\ell_ {1}' -\cdots- k\rangle}.
\end{align*}
By direct calculations, we have the estimation \[\sum_{k\in \llbracket 1,M \rrbracket}\frac{1}{\langle j_1'+\cdots+j_{r'}'-\ell_ {1}' -\cdots- k\rangle} \leq \sum_{k\in \llbracket 1,M \rrbracket}\frac{1}{\langle k \rangle} \lesssim \log M.\]
As a result, taking the norm of the Poisson bracket we obtain
\[{\Vert\{H,\chi\}\Vert}_{\mathscr{H}}= \sup_{j'',\ell''\in \llbracket 1,M \rrbracket^{r''}}\abs{N_{j'',\ell''}}\lesssim_{r,r'} \log M {\Vert H \Vert}_{\mathscr{H}}{\Vert \chi \Vert}_{\mathscr{C}}.\qedhere\]
\end{proof}

 \section{Birkhoff normal form theorem}\label{z3}



 Now, we present Birkhoff normal form theorem in low regularity developed by Bernier and Grébert in \cite{bernier2021birkhoff} and provide a rigorous proof following the techniques from \cite{bernier:hal-03334431}. It plays an essential role to help us prove our main result. To start, consider a polynomial Hamiltonian
\[H: \mathbb{C}^{\llbracket 1,M \rrbracket} \to \mathbb{R} \quad \text{ with }\quad  H=Z_2+P\] where $Z_2$ is a quadratic Hamiltonian of the form $Z_2: \mathbb{C}^{\llbracket 1,M \rrbracket} \to \mathbb{R}$ written as
\[ Z_2(u) = \frac{1}{2} \sum_{j \in \llbracket 1,M \rrbracket} w_{j}\abs{u_j}^2 \]
with $w_{j}=\Lambda_{j,V}$ being the frequencies of (\fcolorbox{red}{ white}{\ref{N1}}) \textcolor{black}{(the eigenvalues of operator T+V which we saw in Proposition \fcolorbox{red}{ white}{\ref{N7}} remain close to those of $T$)} satisfying the non-resonant condition (in the sense of Proposition \fcolorbox{red}{ white}{\ref{po1}}). Moreover, $P$ is a polynomial Hamiltonian of degree $2p+2\geq 4$ satisfying 
\[P \in \mathscr{H}^{2p+2}_M \quad \text{and} \quad {\Vert P \Vert}_{\mathscr{H}} \leq C_0 \]   
\textcolor{black}{for some $C_0 >0$}. Then the theorem writes:

 \begin{thm}\label{N102}
Let $r\geq 1$ and $N\geq 1$. There exist two positive constants $C$ \textcolor{black}{depending on $(C_0,r,\beta_{r,N})$ where $\beta_{r,N}$ is the constant given in Proposition \fcolorbox{red}{ white}{\ref{po1}}} and $b$ \textcolor{black}{depending on $r$}, such that for every $M\geq 2$ and every polynomial Hamiltonian $H$ described above,  we can find $\varepsilon_0 \geq 1/(C(\log M)^b)$ and two smooth symplectic maps $\tau^{(0)}$ and $\tau^{(1)}$ defined on $B_{\mathbb{C}^{\llbracket 1,M \rrbracket}}(0,\varepsilon_0)$ and $B_{\mathbb{C}^{\llbracket 1,M \rrbracket}}(0,2\varepsilon_0)$ respectively, satisfying the "close to the identity property"
\begin{align} \label{eq81}
\forall \sigma \in \{0,1\}, \,{\Vert u \Vert}_{h^{1/2}} < 2^{\sigma}\varepsilon_0 &\implies {\Vert\tau^{(\sigma)}(u) - u\Vert}_{h^{1/2}} \leq \left( \frac{{\Vert u \Vert}_{h^{1/2}}} {2^{\sigma}\varepsilon_0} \right)^{2p}{\Vert u \Vert}_{h^{1/2}}
\end{align}
and making the following diagram to commute
\[\begin{tikzcd}
B_{\mathbb{C}^{\llbracket 1,M \rrbracket}}(0,\varepsilon_0) \rar{\tau^{(0)}}\arrow[black, bend right]{rr}[black,swap]{\normalfont{\mathrm{id}}_{\mathbb{C}^{\llbracket 1,M \rrbracket}}}  & B_{\mathbb{C}^{\llbracket 1,M \rrbracket}} (0,2\varepsilon_0) \rar{\tau^{(1)}}  & \mathbb{C}^{\llbracket 1,M \rrbracket}
\end{tikzcd}\] 
such that $(Z_2 +P)\circ \tau^{(1)}$ admits on $ B_{\mathbb{C}^{\llbracket 1,M \rrbracket}} (0,2\varepsilon_0)$ the following decomposition
\begin{align}\label{ikea104}
(Z_2 +P)\circ \tau^{(1)} &= Z_2 + \underbrace{Q^{(2p+2)}+ \cdots + Q^{(2r+2p)}}_{:=Q}+R \end{align}
where $Q$ is a polynomial of degree $2(r+p)$ commuting with the low actions given by $I_{q}(u) = \abs{u_{q}}^2$ with $q \leq N$. In other words, we have the property
\[ \forall q \geq 1,\hspace{0.1cm} q \leq N \implies \{I_{q}, Q\} =0. \]
Besides, the remainder term $R$ is a smooth function on $B_{\mathbb{C}^{\llbracket 1,M \rrbracket}}(0,2\varepsilon_0)$ satisfying 
\[ {\Vert\nabla R(u)\Vert}_{h^{-1/2}} \leq C(\log M)^b{\Vert u\Vert}_{h^{1/2}}^{2r+2p}.\] 
Moreover, for $\sigma \in \{0,1\}$ and $u \in B_{\mathbb{C}^{\llbracket 1,M \rrbracket}}(0,2^{\sigma}\varepsilon_0),$ $\mathrm{d}\tau^{(\sigma)}(u)$ satisfies the bounds
\begin{align}\label{eq82} 
{\Vert \mathrm{d}\tau^{(\sigma)}(u)\Vert}_{\mathscr{L}(h^{1/2})} &\leq 4^{r} \quad\text{and}\quad {\Vert \mathrm{d}\tau^{(\sigma)}(u)\Vert}_{\mathscr{L}(h^{-1/2})} \leq 4^{r}.\end{align}
 \end{thm}
 
\begin{proof}
We proceed with the proof by induction on $r_* \in \llbracket p+1,r+p+1 \rrbracket$. We assume that there exists two positive constants $C$ and $b$, such that for every $M\geq 2$ and every polynomial Hamiltonian $H$, we can find $\varepsilon_0 \geq 1/(C(\log M)^b)$ and two smooth, symplectic and close to the identity maps $\tau^{(0)}$ and $\tau^{(1)}$ making the above diagram commute such that $(Z_2 +P)\circ \tau^{(1)}$ admits on $ B_{\mathbb{C}^{\llbracket 1,M \rrbracket}}(0,2\varepsilon_0)$ the following decomposition
\[(Z_2 +P)\circ \tau^{(1)} = Z_2 +Q^{(2p+2)}+ \cdots + Q^{(2r+2p)} +R \]
where every $Q^{(n)}\in {\mathscr{H}}^{n}_M$ is a polynomial Hamiltonian of degree $n$ satisfying ${\Vert Q^{(n)} \Vert}_{\mathscr{H}} \leq C(\log M)^b$ and commuting with the low actions, i.e.
\[  \forall n < 2r_*, \forall q \geq 1, \hspace{0.1cm} q \leq N \implies \{I_{q}, Q^{(n)}\} =0. \]
Besides, the remainder term $R$ is a smooth function on $B_{\mathbb{C}^{\llbracket 1,M \rrbracket}}(0,2\varepsilon_0)$ satisfying 
\[ {\Vert\nabla R(u)\Vert}_{h^{-1/2}} \leq C(\log M)^b{\Vert u\Vert}_{h^{1/2}}^{2r+2p}.\] 
Moreover, for $\sigma \in \{0,1\}$ and $u \in B_{\mathbb{C}^{\llbracket 1,M \rrbracket}}(0,2^{\sigma}\varepsilon_0),$ $\mathrm{d}\tau^{(\sigma)}(u)$ satisfies the bounds
\begin{align*}
{\Vert \mathrm{d}\tau^{(\sigma)}(u)\Vert}_{\mathscr{L}(h^{1/2})} &\leq 4^{r_*-p-1} \quad\text{and}\quad {\Vert \mathrm{d}\tau^{(\sigma)}(u)\Vert}_{\mathscr{L}(h^{-1/2})} \leq 4^{r_*-p-1}.
\end{align*}
Notice that for $r_*=p+1,$ we have nothing to do and the proof is direct. Indeed, we can set 
\[\tau^{(0)} = \tau^{(1)} = \mathrm{id}_{\mathbb{C}^{\llbracket 1,M \rrbracket}}, \quad R=0, \quad b=0, \quad \textcolor{black}{Q^{(2p+2)}= P \text{ and } Q^{(n)}=0 \text{ for } n>2p+2}.\]
We turn to the induction step. Note that in order to avoid confusion, we will distinguish between the terms associated to $r_*$ and the ones associated to $r_*+1$ by a symbol $\sharp$. Now, we begin with the work. \vspace{0.3cm}\\ 
$\bullet$ \textcolor{black}{\textbf{Step 1}:} we will start by decomposing $Q^{(2r_*)}.$ Our goal is to write $Q^{(2r_*)}$ as $L+U$ where $L,U \in \mathscr{H}^{2r_*}_M$ and U commutes with the low actions. For this, we recall (\fcolorbox{red}{ white}{\ref{kappa}}) and define 
\[L_{j,\ell} = \begin{cases} Q_{j,\ell}^{(2r_*)} & \text{if } \kappa(j,\ell) \leq N, \\
           0 & \text{otherwise} \end{cases}
\quad \text{and} \quad  U_{j,\ell} = \begin{cases} 0 & \text{if } \kappa(j,\ell) \leq N, \\
           Q_{j,\ell}^{(2r_*)} & \text{otherwise},
           \end{cases}\]
and we check that $U$ commutes with $I_{q}.$ Using direct calculations, we get 
\[ \{I_{q},U\} = 2i \sum_{j,\ell\in \llbracket 1,M \rrbracket^{r_*}} \sum_{n=1}^{r_*}\left(\mathds{1}_{j_n=q} - \mathds{1}_{\ell_n=q} \right)U_{j,\ell} u_{j_1} \cdots u_{j_{r_*}} \overline{u_{\ell_{1}}}\cdots  \overline{u_{\ell_{r_*}}}.\]
By definition of $U_{j,\ell}$ and $\kappa(j,\ell)$ \textcolor{black}{(see (\fcolorbox{red}{ white}{\ref{kappa}}))}, it is obvious that for $q \leq N < \kappa(j,\ell)$ we have \[ \sum_{n=1}^{r_*}\left(\mathds{1}_{j_n=q} - \mathds{1}_{\ell_n=q} \right) =0.\]
$\bullet$ \textcolor{black}{\textbf{Step 2}:} we choose a Hamiltonian $\chi$ in such a way that $L$, the remaining terms of $Q^{(2r_*)}$, vanish by solving the following cohomological equation:
\begin{align}\label{coh1} \{ \chi,Z_2\} +L =0.\end{align}
To seek in, we recall $\Omega_{j,\ell}(V):= w_{j_1}+\cdots +w_{j_{r_*}} - w_{\ell_1}-\cdots -w_{\ell_{r_*}}$ and let $\chi \in \mathscr{H}_M^{2r_*}$ be the Hamiltonian defined by 
\[\chi_{j,\ell}= \begin{cases} \frac{L_{j,\ell}}{i\Omega_{j,\ell}(V) } & \text{if } \kappa(j,\ell) \leq N, \\
           0 & \text{otherwise}. \end{cases}\]
Using direct computations, we can verify that $\chi$ satisfies (\fcolorbox{red}{ white}{\ref{coh1}}). Moreover, we have a good control of its $\mathscr{C}$-norm. Indeed, since the frequencies are non-resonant (recall Proposition \fcolorbox{red}{ white}{\ref{po1}}), there exists $\beta_{r_*,N}\in (0,1)$ and such that 
\[\kappa(j,\ell) \leq N \quad \implies \quad \Omega_{j,\ell}(V) \geq \beta_{r_*,N} =:\delta.\]
Consequently, using Lemma \fcolorbox{red}{ white}{\ref{ike1}} we get
\begin{align*}
    \frac{\langle j_1 + \cdots + j_{r_*} -\ell_{1} - \cdots - \ell_{r_*} \rangle}{\abs{\Omega_{j,\ell}(V)}}
    &\leq \frac{(r_*+1)C'}{\abs{\Omega_{j,\ell}(V)}} +1 \leq (r_*+1)C'\delta^{-1}+1.
    \end{align*} 
Thus, dividing by $\langle j_1 + \cdots + j_{r_*} -\ell_{1} - \cdots - \ell_{r_*} \rangle$ and using the fact that $\delta<1$ we get 
\begin{align*}
     \frac{1}{\abs{\Omega_{j,\ell}(V)}} &\leq \frac{(r_*+1)C'\delta^{-1}}{\langle j_1 + \cdots + j_{r_*} -\ell_{1} - \cdots - \ell_{r_*} \rangle}+\frac{\delta^{-1}}{\langle j_1 + \cdots + j_{r_*} -\ell_{1} - \cdots - \ell_{r_*} \rangle}\\
     &\leq \frac{(r_*+2)C'\delta^{-1}}{\langle j_1 + \cdots + j_{r_*} -\ell_{1} - \cdots - \ell_{r_*} \rangle}.
\end{align*}
Therefore, we obtain 
\[\abs{\chi_{j,\ell}} = \abs{\frac{L_{j,\ell}}{\Omega_{j,\ell}(V)}}\lesssim \frac{\abs{L_{j,\ell}}(r_*+2)\delta^{-1}}{\langle j_1 + \cdots + j_{r_*} -\ell_{1} - \cdots - \ell_{r_*} \rangle}.\]
By construction, we know that $L$ satisfies the same norm estimate as $Q^{(2r_*)}$. So, taking the $\sup$ and using the induction hypothesis on ${\Vert Q^{(2r_*)}\Vert}_{\mathscr{H}}$ we establish that
\begin{align*}
    {\Vert \chi \Vert}_{\mathscr{C}} &= \sup_{j,\ell\in \llbracket 1,M \rrbracket^{r_*}}\abs{\chi_{j,\ell}} \langle j_1 + \cdots + j_{r_*} -\ell_{1} - \cdots - \ell_{r_*} \rangle \lesssim_{r_*}\delta^{-1}\underbrace{\sup_{j,\ell \in \llbracket 1,M \rrbracket^{r_*}}\abs{L_{j,\ell}}}_{{\Vert L\Vert}_{\mathscr{H}}} \lesssim_{r_*} \delta^{-1} C(\log M)^b.
\end{align*}
$\bullet$ \textcolor{black}{\textbf{Step 3}:} we define the new variables by composing $\tau^{(0)}$ and $\tau^{(1)}$ with the flow of the Hamiltonian $\chi.$ Applying Lemma \fcolorbox{red}{ white}{\ref{imp1}}, we get
\begin{align*}
\varepsilon_1' &= \left(K_1'(\log M)^{(2r_*-1)/2}{\Vert\chi\Vert}_{\mathscr{C}}\right)^{-1/{(2r_*-2)}}
    \end{align*}
where $K_1'>0$ depends on $r_*$ and  a smooth symplectic invertible close to the identity map 
\[ \phi_{\chi}: \Bigg \{ \begin{array}{ccc}  [-1,1]\times B_{\mathbb{C}^{\llbracket 1,M \rrbracket}}(0,\varepsilon_1') & \to &  \mathbb{C}^{\llbracket 1,M \rrbracket} \\ (t,u) &\mapsto & \phi_{\chi}^t (u) \end{array} \]
solving the equation $-i\partial_t \phi_{\chi} = (\nabla\chi)\circ \phi_{\chi}.$ Next, since $ {\Vert \chi \Vert}_{\mathscr{C}}\lesssim_{r_*} \delta^{-1} C(\log M)^b,$ we have 
\begin{align}\label{est1}
\varepsilon_1' &\geq \left(K_2'C (\log M)^{(2r_*-1)/2+b }\right)^ {-1/{ (2r_*-2)}}\geq 6(C_{\sharp}(\log M)^{b_\sharp})^{-1} =: 6\varepsilon_0^{\sharp} \end{align}
where we set $C_{\sharp} \geq 6\max\left((K_2'C)^{1/(2r_*-2)},C\right)$ and $b_\sharp\geq \max \left( b, \frac{1}{2r_*-2}\left( \frac{2r_*-1}{2} +b \right) \right).$ As a consequence, it makes sense to define the maps as mentioned above by
\[ \tau^{(1)}_{\sharp}:=\tau^{(1)} \circ \phi_{\chi}^1 \text{ on } B_{\mathbb{C}^{\llbracket 1,M \rrbracket}}(0,2\varepsilon_0^{\sharp})\hspace{0.2cm} \text{ and } \hspace{0.2cm} \tau^{(0)}_{\sharp}:= \phi_{\chi}^{-1} \circ \tau^{(0)} \text{ on } B_{\mathbb{C}^{\llbracket 1,M \rrbracket}}(0,\varepsilon_0^{\sharp})\] \textcolor{black}{where the choice of $\varepsilon_0^{\sharp}$ is dependent on the domains of definition of $\tau^{(0)}_{\sharp}$ and $\tau^{(1)}_{\sharp}.$} It is easy to see that the two maps are smooth and symplectic. To check that they are close to the identity, consider $u\in B_{\mathbb{C}^{\llbracket 1,M \rrbracket}}(0,2\varepsilon_0^{\sharp}).$ Then, we have 
\begin{align*}
    {\Vert \tau^{(1)}_{\sharp} u - u \Vert}_{h^{1/2}} &= {\Vert \tau^{(1)}\circ \phi_{\chi}^1 ( u) - u\Vert }_{h^{1/2}}\\
    &= {\Vert \tau^{(1)}\circ \phi_{\chi}^1  (u)- \phi_{\chi}^1  (u) +\phi_{\chi}^1  (u) - u\Vert }_{h^{1/2}}\\
    &\leq \underbrace{{\Vert \tau^{(1)}\circ \phi_{\chi}^1  (u)- \phi_{\chi}^1  (u) \Vert }_{h^{1/2}}}_{=:A_1} + \underbrace{{\Vert \phi_{\chi}^1  (u) - u\Vert }_{h^{1/2}}}_{=:A_2}.
\end{align*}
\textcolor{black}{Moreover, (\fcolorbox{red}{ white}{\ref{est1}}) implies that ${\Vert u \Vert}_{h^{1/2}}\leq 2\varepsilon_0^{\sharp} < 6\varepsilon_0^{\sharp} \leq \varepsilon_1'$. We pass now to estimate $A_1$ and $A_2$.}\\
\underline{\text{Estimate of }$A_1$:}
Since $\phi_{\chi}^1$ is close to the identity, we get
\[{\Vert \phi_{\chi}^t (u) - u\Vert}_{h^{1/2}} \leq \left(\frac{{\Vert u \Vert}_{h^{1/2}}} {\varepsilon_1'}\right)^{2r_*-2}{\Vert u \Vert}_{h^{1/2}} \leq {\Vert u \Vert}_{h^{1/2}}.\]
Also, using the definitions of $\varepsilon_0$ and $\varepsilon_0^{\sharp}$ we obtain that $3\varepsilon_0^{\sharp} \leq \varepsilon_0$, and we establish
\[{\Vert \phi_{\chi}^t (u)\Vert}_{h^{1/2}} \leq 2{\Vert u \Vert}_{h^{1/2}} \leq 4 \varepsilon_0^{\sharp} < 6 \varepsilon_0^{\sharp} \leq 2\varepsilon_0. \]
By induction hypothesis, we know that $\tau^{(1)}$ is close to the identity (see (\fcolorbox{red}{ white}{\ref{eq81}})), thus
\begin{align}
\nonumber{\Vert\tau^{(1)}\circ \phi_{\chi}^1(u) - \phi_{\chi}^1(u)\Vert} _{h^{1/2}} \nonumber&\leq \left( \frac{{\Vert \phi_{\chi}^1(u) \Vert}_{h^{1/2}}} {2\varepsilon_0} \right)^{2p}{\Vert \phi_{\chi}^1(u) \Vert}_{h^{1/2}}\leq 2 \left( \frac{{\Vert u \Vert}_{h^{1/2}}} {6\varepsilon_0^{\sharp}} \right)^{2p}{\Vert u \Vert}_{h^{1/2}}\\
&\leq \label{pl909}\frac{2}{3} \left( \frac{{\Vert u \Vert}_{h^{1/2}}} {2\varepsilon_0^{\sharp}} \right)^{2p}{\Vert u \Vert}_{h^{1/2}}.
\end{align}
\underline{\text{Estimate of }$A_2$:} We write
\begin{align}
    \nonumber {\Vert \phi_{\chi}^t (u) - u\Vert}_{h^{1/2}} &\leq \left(\frac{{\Vert u \Vert}_{h^{1/2}}} {\varepsilon_1'}\right)^{2r_*-2}{\Vert u \Vert}_{h^{1/2}} \leq \left(\frac{{\Vert u \Vert}_{h^{1/2}}} {6\varepsilon_0^{\sharp}}\right)^{2r_*-2}{\Vert u \Vert}_{h^{1/2}}\\ &\leq \label{pl2}\frac{1}{3}\left(\frac{{\Vert u \Vert}_{h^{1/2}}} {2\varepsilon_0^{\sharp}}\right)^{2r_*-2}{\Vert u \Vert}_{h^{1/2}}.
\end{align}
Finally, replacing (\fcolorbox{red}{ white}{\ref{pl909}}) and (\fcolorbox{red}{ white}{\ref{pl2}}) back and noting that $2r_*-2>2p$, we obtain
\[{\Vert \tau^{(1)}_{\sharp} u - u \Vert}_{h^{1/2}} \leq \left(\frac{2}{3}+\frac{1}{3}\right) \left(\frac{{\Vert u \Vert}_{h^{1/2}}} {2\varepsilon_0^{\sharp}}\right)^{2p}{\Vert u \Vert}_{h^{1/2}} \leq \left(\frac{{\Vert u \Vert}_{h^{1/2}}} {2\varepsilon_0^{\sharp}}\right)^{2p}{\Vert u \Vert}_{h^{1/2}}.\]
Same arguments and estimations can be used to prove this result for the map $\tau^{(0)}_{\sharp}.$
It remains to prove that these two maps make the diagram commutative. For this, take $u\in B_{\mathbb{C}^{\llbracket 1,M \rrbracket}}(0,\varepsilon_0^{\sharp}).$ Since $\tau_{\sharp}^{(0)}$ is close to the identity, then we have
\[\phi_{\chi}^{-1}\circ \tau^{(0)}(u) = \tau_{\sharp}^{(0)}(u) \in B_{\mathbb{C}^{\llbracket 1,M \rrbracket}}(0,2\varepsilon_0^{\sharp}) \subset B_{\mathbb{C}^{\llbracket 1,M \rrbracket}}(0,\varepsilon_1').\] 
Thus, since $\phi_{\chi}^1$ is invertible, we obtain 
\[\tau_{\sharp}^{(1)}\circ \tau_{\sharp}^{(0)}(u) = \tau^{(1)}\circ \phi_{\chi}^1 \circ \phi_{\chi}^{-1} \circ \tau^{(0)}(u) = \tau^{(1)}\circ \tau^{(0)}(u) = \mathrm{id}_{\mathbb{C}^{\llbracket 1,M \rrbracket}}. \]
$\bullet$ \textcolor{black}{\textbf{Step 4}:} our goal now is to decompose $(Z_2 +P)\circ \tau^{(1)}_{\sharp}$ on $B_{\mathbb{C}^{\llbracket 1,M \rrbracket}}(0,2\varepsilon_0^{\sharp}).$ Notice that by definition of $\tau^{(1)}_{\sharp}$ and using induction hypothesis, we have
\begin{align*}
    (Z_2+P)\circ \tau^{(1)}_{\sharp} &= (Z_2+P)\circ \tau^{(1)}\circ \phi_{\chi}^1 
    = Z_2\circ \phi_{\chi}^1 +\sum_{n=2p+2}^{2r+2p}Q^{(n)}\circ \phi_{\chi}^1+R\circ \phi_{\chi}^1.
\end{align*}
Now since $\phi_{\chi}^t$ is a smooth function, applying Taylor expansion between $0$ and $1$ gives
\begin{align*}
     &(Z_2+P)\circ \tau^{(1)}_{\sharp}\\
     &= Z_2 +\{\chi,Z_2\} +\sum_{k=2}^{m_{r_*}+1} \frac{1}{k!} \mathrm{ad}_\chi^k Z_2 +  \int _0^1 \frac{(1-t)^{m_{r_*}+1}}{(m_{r_*}+1)!}\mathrm{ad}_\chi^{m_{r_*}+2} Z_2\circ \phi_{\chi}^t  \, \mathrm{d}t \\
     &\hspace{1.5cm}+ \sum_{n=2p+2}^{2r+2p}\Big[ Q^{(n)} + \sum_{k=1}^{m_n} \frac{1}{k!} \mathrm{ad}_\chi^k Q^{(n)}+  \int _0^1 \frac{(1-t)^{m_n}}{m_n!}\mathrm{ad}_\chi^{m_n+1} Q^{(n)}\circ \phi_{\chi}^t   \, \mathrm{d}t\Big]+ R\circ \phi_{\chi}^1
\end{align*}
with $m_n$ the largest integer such that $n+ m_n(2r_* -2)<2r+2p+2.$ From (\fcolorbox{red}{ white}{\ref{coh1}}) we have
$$\mathrm{ad}_\chi^{k+1} Z_2 = \underbrace{\{\chi,\{\chi,\cdots, \{\chi,Z_2\}\cdots \} \}}_{k+1 \text{ times}} = - \underbrace{\{\chi,\{\chi,\cdots, \{\chi,L\}\cdots \} \}}_{k \text{ times}} = -\mathrm{ad}_\chi^{k} L.$$ So, we write 
\begin{align*}
    &(Z_2+P)\circ \tau^{(1)}_{\sharp}\\
    &=  Z_2 +\{\chi,Z_2\} -\sum_{k=1}^{m_{r_*}} \frac{1}{(k+1)!} \mathrm{ad}_\chi^{k} L -  \int _0^1 \frac{(1-t)^{m_{r_*}+1} }{(m_{r_*}+1)!}\mathrm{ad}_\chi^{m_{r_*}+1} L\circ \phi_{\chi}^t   \, \mathrm{d}t \\
     &+ \sum_{n=2p+2}^{2r+2p}\Big[ Q^{(n)} + \sum_{k=1}^{m_n} \frac{1}{k!} \mathrm{ad}_\chi^{k} Q^{(n)} +  \int _0^1 \frac{ (1-t)^{m_n}}{m_n!}\mathrm{ad}_\chi^{m_n+1} Q^{(n)}\circ \phi_{\chi}^t  \, \mathrm{d}t\Big]+ R\circ \phi_{\chi}^1\\
     &= Z_2 + \sum_{n=2p+2}^{2r_*} Q^{(n)} + \{\chi,Z_2\} + \sum_{n=2r_*+1}^{2r+2p} Q^{(n)} +\sum_{n=2p+2}^{2r+2p} \sum_{k=1}^{m_n} \frac{1}{k!} \mathrm{ad}_\chi^{k} Q^{(n)} -\sum_{k=1} ^{m_{r_*}} \frac{1}{(k+1)!} \mathrm{ad}_\chi^{k} L \\  
     &+ R\circ \phi_{\chi}^1 -  \int _0^1 \frac{(1-t)^{m_{r_*}+1} }{(m_{r_*}+1)!}\mathrm{ad}_\chi^{m_{r_*}+1} L\circ \phi_{\chi}^t   \, \mathrm{d}t+ \sum_{n=2p+2}^{2r+2p} \int _0^1 \frac{ (1-t)^{m_n}}{m_n!}\mathrm{ad}_\chi^{m_n+1} Q^{(n)}\circ \phi_{\chi}^t  \, \mathrm{d}t.
\end{align*}
Using the induction hypothesis and Proposition \fcolorbox{red}{ white}{\ref{1209}}, it is easy to see that $Q^{(n)}$ is of order $n,$ $\{\chi,Z_2\}$ is of order $2r_*,$ $\mathrm{ad}_\chi^{k} Q^{(n)}$ is of order $n+2k(r_*-1)>2r_*$ and $\mathrm{ad}_\chi^{k} L$ is of order $2r_* +2k(r_*-1)>2r_*.$ As a result, after reordering it \textcolor{black}{makes sense to set}:
 \begin{align*}
  \text{for}\hspace{0.2cm} n < 2r_*,\hspace{0.2cm}   &Q^{(n)}_{\sharp}  =   Q^{(n)},\\
\text{for}\hspace{0.2cm} n=2r_*,\hspace{0.2cm} &Q^{(n)}_{\sharp} =  Q^{(n)}+ \{\chi,Z_2\},\\
\text{for}\hspace{0.2cm} n>2 r_*,\hspace{0.2cm} &Q^{(n)}_{\sharp}= \sum_{\substack{n_*, k \\ n_*+2k(r_*-1)=n}} \frac{1}{k!} \mathrm{ad}_\chi^{k} Q^{(n_*)} -\sum_{\substack{k \\2r_* +2k(r_*-1)=n}} \frac{1}{(k+1)!} \mathrm{ad}_\chi^{k} L,
  \end{align*}
and \begin{align}\label{Rsharp}
R_{\sharp} &= R\circ \phi_{\chi}^1 -  \int _0^1\left( \frac{(1-t)^{m_{r_*}+1} }{(m_{r_*}+1)!}\mathrm{ad}_\chi^{m_{r_*}+1} L\circ \phi_{\chi}^t - \sum_{n=2p+2}^{2r+2p} \frac{ (1-t)^{m_n}}{m_n!}\mathrm{ad}_\chi^{m_n+1} Q^{(n)}\circ \phi_{\chi}^t\right)  \, \mathrm{d}t.\end{align}
Notice that $Q^{(2r_*)}_{\sharp} =Q^{(2r_*)}+ \{\chi,Z_2\} = Q^{(2r_*)} -L = U$ which commutes with the low actions by construction (we already checked this property in the beginning of the proof). Hence, for $n\leq 2r_*$, $Q^{(n)}_{\sharp} \in \mathscr{H}_M^{n}$ commutes with the low actions and we have 
\[n< 2(r_*+1) \quad\text{and}\quad  q \leq N \implies \{I_{q}, Q^{(n)}_{\sharp}\} =0.\]
Moreover, we have the bound \[ {\Vert Q_{\sharp}^{(n)}\Vert}_{\mathscr{H}}\leq{\Vert Q^{(n)}\Vert}_{\mathscr{H}} \leq C(\log M)^b.\] 
For $n>2r_*,$ we use Proposition \fcolorbox{red}{ white}{\ref{1209}} and the estimate on ${\Vert \chi\Vert}_{\mathscr{C}}$ to obtain that
\begin{align*}
    {\Vert\mathrm{ad}_\chi^{k} Q^{(n_*)}\Vert}_{\mathscr{H}}&\lesssim_{r} (\log M)^k {\Vert\chi\Vert}_{\mathscr{C}}^k {\Vert Q^{(n_*)}\Vert}_{\mathscr{H}} \lesssim_r \beta_{r_*,N}^{-k} C^{k+1}(\log M)^{k+ b(k+1)}.
     \end{align*}
Since $\mathrm{ad}_\chi^{k} L$ and $\mathrm{ad}_\chi^{k} Q^{(n_*)}$ enjoy the same estimate when $2r_* +2k(r_*-1)=n$ and since $k\leq 2r+2p+2$, we deduce that for $C_{\sharp} \gtrsim_r \beta_{r_*,N}^{-2r-2p-2} C^{2r+2p+3}$ and $b_{\sharp} \geq  (2r+2p+2)+ b(2r+2p+3)$
\begin{align*}
    {\Vert Q_{\sharp}^{(n)}\Vert}_{\mathscr{H}} &\leq C_{\sharp}(\log M)^{b_{\sharp}} .
\end{align*}

$\bullet$ \textcolor{black}{\textbf{Step 5}:} we still have to control the remainder term.  For this we fix $u\in B_{\mathbb{C}^{\llbracket 1,M \rrbracket}}(0,2\varepsilon_0^{\sharp})$ and start by checking that $\nabla(R \circ \phi_{\chi}^1) \in h^{-1/2}.$ By composition,  we have \[ \nabla(R \circ \phi_{\chi}^1)(u) = (\mathrm{d}\phi_{\chi}^1(u))^*(\nabla R)\circ \phi_{\chi}^1(u).\] We know from Lemma \fcolorbox{red}{ white}{\ref{imp1}} that $ {\Vert (\mathrm{d}\phi_{\chi}^1(u))^*\Vert}_{\mathscr{L}(h^{-1/2})} = {\Vert \mathrm{d}\phi_{\chi}^1(u)\Vert}_{\mathscr{L}(h^{1/2})}\leq 2$. Also, since $(\nabla R) \circ \phi_{\chi}^1 \in h^{-1/2}$ , then $(\mathrm{d}\phi_{\chi}^1)^*(\nabla R)\circ \phi_{\chi}^1 \in h^{-1/2}$. Now we turn to controlling this term in $h^{-1/2}$. Using the induction hypothesis and ${\Vert \phi_{\chi}^1(u)\Vert}_{h^{1/2}}\leq 2{\Vert u \Vert}_{h^{1/2}}$, we get 
\begin{align*}
   {\Vert \nabla(R \circ \phi_{\chi}^1)(u) \Vert}_{h^{-1/2}} &=  {\Vert (\mathrm{d}\phi_{\chi}^1(u))^*(\nabla R)\circ \phi_{\chi}^1(u)\Vert}_{h^{-1/2}} \\
   &\leq {\Vert (\mathrm{d}\phi_{\chi}^1(u))^*\Vert}_{\mathscr{L}(h^{-1/2})} {\Vert(\nabla R)\circ \phi_{\chi}^1(u)\Vert}_{h^{-1/2}}\\
    &\leq 2 C(\log M)^b {\Vert \phi_{\chi}^1(u) \Vert}_{h^{1/2}}^{2r+2p}\\
    &\leq 4^{r+p} C(\log M)^b {\Vert u \Vert}_{h^{1/2}}^{2r+2p}.
\end{align*}
Next, we estimate the terms of $R_{\sharp}$ inside the integral. We denote $r_n:= n+ (m_n+1)(2r_*-2) $, and we notice that $\mathrm{ad}_\chi^{m_n+1} Q^{(n)}$ is a smooth function belonging to $\mathscr{H}^{r_n}$. Thus, arguing as above and using Proposition \fcolorbox{red}{ white}{\ref{PLKI}} and Proposition \fcolorbox{red}{ white}{\ref{1209}}, we notice that we have for $2p+2 \leq n\leq 2r+2p$ and $t\in [0,1]$ we establish
\begin{align*} 
&{\Vert\nabla(\mathrm{ad}_\chi^{m_n+1} Q^{(n)}\circ \phi_{\chi}^t)(u) \Vert}_{h^{-1/2}} \\
&\leq 2 {\Vert\nabla(\mathrm{ad}_\chi^{m_n+1} Q^{(n)})\circ \phi_{\chi}^t(u)\Vert}_{h^{-1/2}}\\
&\lesssim_r (\log M)^{r_n/2} {\Vert \mathrm{ad}_\chi^{m_n+1} Q^{(n)}) \Vert}_{\mathscr{H}} {\Vert \phi_{\chi}^t(u) \Vert}_{h^{1/2}}^{r_n-1}\\
&\lesssim_r (\log M)^{r_n/2}(\log M)^{m_n+1}{\Vert Q^{(n)}\Vert}_{\mathscr{H}}{\Vert\chi\Vert}_{\mathscr{C}}^{m_n+1} {\Vert \phi_{\chi}^t(u) \Vert}_{h^{1/2}}^{r_n-1}\\
&\lesssim_r (\log M)^{r_n/2}(\delta^{-1}\log M)^{m_n+1}(C(\log M)^b)^{m_n+2} {\Vert \phi_{\chi}^t(u) \Vert}_{h^{1/2}}^{r_n-1}  \\
&\lesssim_r (\beta_{r_*,N})^{-m_n-1}C^{m_n+2}(\log M)^{m_n+1}(\log M)^{r_n/2}(\log M)^{b(m_n+2)} {\Vert \phi_{\chi}^t(u) \Vert}_{h^{1/2}}^{r_n-1}.
\end{align*}
\textcolor{black}{Recall that $m_n$ is the largest integer such that $n+ m_n(2r_* -2)<2r+2p+2.$} Now, using the fact that $m_n \leq 2r+2p+2,$ $r_n \in \llbracket 2(r+p+1),4r+4p+2) \rrbracket$ and ${\Vert u \Vert}_{h^{1/2}} \leq 2,$ we have
\[{\Vert\nabla(\mathrm{ad}_\chi^{m_n+1} Q^{(n)}\circ \phi_{\chi}^t)(u) \Vert}_{h^{-1/2}} \lesssim_r C_{\sharp} (\log M)^{b_{\sharp}} {\Vert u \Vert}_{h^{1/2}}^{2r+2p}\]
where $C_{\sharp} \gtrsim_r (\beta_{r_*,N})^{-2r-2p-3}C^{2r+2p+4}$ and $b_{\sharp} \geq 4(r+p+1)+ 2b(r+p+2).$ We can also check using similar calculations that ${\Vert\nabla(\mathrm{ad}_\chi^{m_{r_*}+1} L \circ \phi_{\chi}^t)(u) \Vert}_{h^{-1/2}}$ enjoys the same bound as ${\Vert\nabla(\mathrm{ad}_\chi^{m_n+1} Q^{(n)}\circ \phi_{\chi}^t)(u) \Vert}_{h^{-1/2}}$. Hence, putting the results together with \fcolorbox{red}{ white}{\ref{Rsharp}}, we obtain
\begin{align*}
    {\Vert \nabla R_{\sharp}(u)\Vert}_{h^{-1/2}} \leq {\Vert \nabla( R\circ \phi_{\chi}^1)(u)\Vert}_{h^{-1/2}} &+\int _0^1\Big(  \frac{1}{(m_{r_*}+1)!}{\Vert\nabla(\mathrm{ad}_\chi^{m_{r_*}+1} L\circ \phi_{\chi}^t)(u)\Vert}_{h^{-1/2}}\\
    &+  \sum_{n=2p+2}^{2r+2p}\frac{1}{m_n!}{\Vert \nabla(\mathrm{ad}_\chi^{m_n+1} Q^{(n)}\circ \phi_{\chi}^t)(u)\Vert}_{h^{-1/2}} \Big) \,\mathrm{d}t.
\end{align*}
Taking furthermore $C_{\sharp} \gtrsim_r C$ and $b_{\sharp} \geq b$ we get that
\[{\Vert \nabla R_{\sharp}(u)\Vert}_{h^{-1/2}} \lesssim_r C_{\sharp} (\log M)^{b_{\sharp}} {\Vert u \Vert}_{h^{1/2}}^{2r+2p}.\]
In order to end the proof, we choose the most optimal constants and thus we set 
\[C_{\sharp} \simeq_r \max \big((K_2'C)^{1/(2r_*-2)},(\beta_{r_*,N})^{-2r-2p-3}C^{2r+2p+4}\big) \, \text{ and } \, b_{\sharp} = 4(r+p+1)+ 2b(r+p+2).\]
\end{proof}

  \section{Proof of the main result}\label{z4}
  The last part of this paper is dedicated to proving the main result known to be a dynamical corollary of Theorem \fcolorbox{red}{ white}{\ref{N102}}. Before seeking into the details, we will briefly discuss the global well-posedness of our model (in the same spirit, see \cite{cazenave2003semilinear}) which is ensured thanks to the conservation of the Hamiltonian and the mass.
  
  \begin{lem}\label{zara2}
  For all $\rho>0$, there exist $\varepsilon_{\rho}>0$ and $C_{\rho}>0$ such that provided ${\Vert V \Vert}_{L^{\infty}} \leq \rho,$ we have for $u\in \widehat{H}^1$ satisfying ${\Vert u \Vert}_{\widehat{H}^1} \leq C_{\rho} \varepsilon_{\rho}$
  \[ {C_{\rho}}^{-1}  {\Vert u \Vert}_{\widehat{H}^1}^2 \leq H(u) + \rho  {\Vert u \Vert}_{L^2}^2 \leq {C_{\rho}} {\Vert u \Vert}_{\widehat{H}^1}^2 .\] 
  \end{lem}
  
    
   \begin{lem}\label{zara1}
 If $u \in \mathscr{C}^0((-T,T), \widehat{H}^1)$ solves (\fcolorbox{red}{ white}{\ref{N1}}), then its energy and mass are preserved
  \[ \forall t\in (-T,T), \quad H(u(t)) = H(u^{(0)} ) \quad \text{ and } \quad {\Vert u(t) \Vert}_{L^2}^2 =  {\Vert u^{(0)}  \Vert}_{L^2}^2. \]
  \end{lem}  
    The proofs of the above results use Sobolev embeddings and the fact that $\widehat{H}^1$ is an algebra\footnote{This is due to Proposition 2.1.1 in \cite{zbMATH05153585} and the continuous inclusions $ \widehat{H}^1(\mathbb{R}) \subset H^1(\mathbb{R}) \subset L^{\infty}(\mathbb{R})$.}.
  
  
  Consequently, we obtain the global well-posedness of our Schr\"{o}dinger equation:
  \begin{thm}(Global Well-posedness)\label{ber4}
Let $\rho >0$ and $\varepsilon_{\rho}>0$ be given by Lemma \fcolorbox{red}{ white}{\ref{zara2}}. Provided that $\varepsilon:= {\Vert u^{(0)} \Vert}_{\widehat{H}^1} \leq \varepsilon_{\rho}$ and ${\Vert V \Vert}_{\widehat{H}^1} \leq \rho$, there exists a unique global solution $u \in \mathscr{C}_b^0(\mathbb{R}, \widehat{H}^1) \cap \mathscr{C}^1(\mathbb{R}, \widehat{H}^{-1})$ to (\fcolorbox{red}{ white}{\ref{N1}}). 
 \end{thm}

 \begin{proof} The idea of the global well-posedness is quite standard: the local well-posedness is acheived by a fixed point argument. From this, we deduce Theorem \fcolorbox{red}{ white}{\ref{ber4}} by extension using the boundedness of ${\Vert u(t) \Vert}_{\widehat{H}^1}$.
 \end{proof}
  
  Note that Lemma \fcolorbox{red}{ white}{\ref{zara1}} can now be extended for all $t\in \mathbb{R}.$ As a corollary of Lemma \fcolorbox{red}{ white}{\ref{zara2}} and the Hamiltonian and mass conservation, the norm of the solution is bounded for all $ t\in \mathbb{R}$ 
  \[  {\Vert u(t) \Vert}_{\widehat{H}^1}^2 \leq C_{\rho}\left(H(u(t)) +\rho  {\Vert u(t) \Vert}_{L^2}^2 \right) = C_{\rho}\left( H(u^{(0)} ) +\rho  {\Vert u^{(0)} \Vert}_{L^2}^2 \right) \leq  C_{\rho}^2 {\Vert u^{(0)}  \Vert}_{\widehat{H}^1}^2 \leq C_{\rho}^2 \varepsilon^2. \]

\textbf{Proof of Theorem \fcolorbox{red}{ white}{\ref{ikea1000}}}
To start, recall that we have ${\Vert u(t) \Vert}_{\widehat{H}^s} \simeq {\Vert u(t) \Vert}_{h^{s/2}}.$ Now, we consider ${\Vert u^{(0)}  \Vert}_{\widehat{H}^1} \leq \varepsilon_{\rho}$ where $\varepsilon_{\rho}$ is given in Lemma \fcolorbox{red}{ white}{\ref{zara2}}. We focus on the variations of the low modes. We fix \textcolor{black}{$j_* = N$}, and we aim at estimating $\abs{u_{j_*}(t)}^2.$ To apply the Birkhoff Normal Form Theorem (see Theorem  \fcolorbox{red}{ white}{\ref{N102}}), we need to make a truncation up to a level $M$ in order to restrict our work to the finite dimensional situation of the theorem. To this matter, we let $$M = \varepsilon^{-4r+2}.$$ Furthermore, we consider the eigenspaces of $T+V$ 
\[ E_j = \ker (T+V-\Lambda_j) = \text{Span} (\psi_j) \quad \text{where} \quad L^2(\mathbb{R}) = \bigoplus_{j \geq 1} E_j,\] 
and we introduce $\Pi_{\leq M}$ the orthogonal projection on $\bigoplus\limits_{j \leq M} E_j.$ In other words, we denote $\Pi_{\leq M} := \sum\limits_{j\leq M} \Pi_j$ where $\Pi_j$ is the orthogonal projection on $E_j.$ We set $\Pi_{>M}:= \mathrm{Id}_{L^2} - \Pi_{\leq M}$, $u^{\leq M}:= \Pi_{\leq M} u$ and $F^{>M}(t) :=\pm \Pi_{\leq M}\left( -\abs{u^{\leq M}}^{2p} u^{\leq M} +\abs{u}^{2p} u \right).$ Notice that if $u$ solves the Schr\"{o}dinger equation (\fcolorbox{red}{ white}{\ref{N1}}), then $u^{\leq M}$ solves the equation
\begin{align}
\nonumber i\partial_t u^{\leq M} &= \Pi_{\leq M}(i\partial_t u)\\
&= \nonumber \Pi_{\leq M}( (T+V)u \pm \abs{u}^{2p} u) \\
&= \nonumber (T+V)u^{\leq M} \pm\Pi_{\leq M}(\abs{u}^{2p} u)\pm\Pi_{\leq M}\left(\abs{u^{\leq M}}^{2p} u^{\leq M}\right) \mp  \Pi_{\leq M}\left(\abs{u^{\leq M}}^{2p} u^{\leq M}\right) \\
&=\label{ikea101} (T+V)u^{\leq M} \pm\Pi_{\leq M}\left(\abs{u^{\leq M}}^{2p} u^{\leq M}\right) + F^{>M}(t).
\end{align}
Our goal is to ensure that the remainder term $ F^{>M}(t)$ is small in this reduction to finite dimension. Thus, we aim to prove that it is negligible provided that $M$ is large enough (of order $\varepsilon^{-4r+2}$). For this, we write 
\[ {\Vert F^{>M}(t)}\Vert_{L^2} =  {\Big\Vert \Pi_{\leq M}\left( \abs{u}^{2p} u - \abs{u^{\leq M}}^{2p} u^{\leq M} \right) \Big\Vert}_{L^2} \leq  {\Big\Vert \abs{u}^{2p} u - \abs{u^{\leq M}}^{2p} u^{\leq M} \Big\Vert}_{L^2} \]
with $ \abs{u}^{2p} u =  \abs{u - u^{\leq M} +u^{\leq M}}^{2p} (u - u^{\leq M} +u^{\leq M}) =  \abs{u^{\leq M} +u^{> M}}^{2p} ( u^{\leq M} +u^{> M}). $ Then, using the Mean Value Inequality, Holder's Inequality, the Sobolev embeddings $\widehat{H}^1 \hookrightarrow H^1 \hookrightarrow L^{6p}$ and the fact that $H^{1/2}$ injects continuously in $L^6,$ we get
\begin{align*}
{\Vert F^{>M}(t)}\Vert_{L^2} &\leq {\Big\Vert  \abs{u^{\leq M} +u^{> M}}^{2p} ( u^{\leq M} +u^{> M}) - \abs{u^{\leq M}}^{2p} u^{\leq M} \Big\Vert}_{L^2}\\
&\lesssim {\Big\Vert  ( u^{\leq M} +u^{> M} - u^{\leq M}) \left(\abs{u^{\leq M} +u^{>M}}^{2p} +\abs{ u^{\leq M}}^{2p}\right) \Big\Vert}_{L^2}\\
&\lesssim {\Vert u^{>M} \Vert}_{L^6} \left(  {\Big\Vert \abs{u^{\leq M} +u^{>M}}^{2p} \Big\Vert}_{L^3}+ {\Big\Vert \abs{ u^{\leq M}}^{2p} \Big\Vert}_{L^3} \right)\\
&\lesssim  {\Vert u^{>M} \Vert}_{L^6} \left(  {\Vert u^{\leq M} +u^{>M} \Vert}_{L^{6p}}^{2p}+ {\Vert  u^{\leq M} \Vert}_{L^{6p}}^{2p} \right)\\
&\lesssim  {\Vert u^{>M} \Vert}_{\widehat{H}^{1/2}} {\Vert  u \Vert}_{\widehat{H}^1}^{2p}. 
\end{align*}
Since $M >N,$ \textcolor{black}{(this is obvious as $N$ is fixed)} then $M>j_*$ and we have $\abs{u_{j_*}}^2 = \abs{u_{j_*}^{\leq M}}^2.$ Moreover, we obtain
\[  {\Vert u^{>M} \Vert}_{\widehat{H}^{1/2}} \lesssim  M^{-1/2}{\Vert u- u^{\leq M} \Vert}_{h^{1/2}}  \lesssim M^{-1/2}   {\Vert u \Vert}_{\widehat{H}^{1}}.\]
Therefore, recalling that $M = \varepsilon^{-4r+2},$ we deduce that for all $t \in \mathbb{R}$ we get
\[ {\Vert F^{>M}(t)}\Vert_{L^2} \lesssim M^{-1/2}  {\Vert  u \Vert}_{\widehat{H}^1}^{2p+1} \lesssim \varepsilon^{2r+2p}.\]

We are now interested in writing (\fcolorbox{red}{ white}{\ref{ikea101}}) as a Hamiltonian system. Indeed, since ${(\psi_j)}_{j \geq 1}$ is a basis of $L^2$, we can identify $\bigoplus\limits_{j \leq M} E_j$ with $\mathbb{R}^{\llbracket 1,M \rrbracket},$ and we can easily check that equation (\fcolorbox{red}{ white}{\ref{ikea101}}) can be written as  
 \[i \partial_t u^{\leq M} =  \nabla H(u^{\leq M}) + F^{>M}(t) \quad \text{where} \quad H = \underbrace{Z_2}_{\text{linear part}}+\underbrace{P}_{\text{perturbation}}.\] 
 In particular, if we express $u$ in terms of the eigenfunctions as $\sum\limits_{j \in \llbracket 1,M \rrbracket}u_j \psi_j,$ we obtain
 \begin{align*}
H(u)\nonumber &= \frac{1}{2} \int_{\mathbb{R}} \abs{\partial_x u}^2 + x^2\abs{u}^2 + V\abs{u}^2  \, \mathrm{d}x \pm \frac{1}{2p+2}\int_{\mathbb{R}} \abs{u}^{2p+2}  \, \mathrm{d}x\\
&= \label{sp1}\frac{1}{2} \int_{\mathbb{R}}\overline{u}(T+V)u  \, \mathrm{d}x + \frac{1}{2p+2}\int_{\mathbb{R}} \abs{u}^{2p+2}  \, \mathrm{d}x\\
&= \frac{1}{2} \int_{\mathbb{R}}\big(\sum_{j \in \llbracket 1,M \rrbracket} \overline{u_j}\psi_j\big)(T+V)\big(\sum _{j\in \llbracket 1,M \rrbracket} u_j\psi_j \big) \, \mathrm{d}x + \frac{1}{2p+2}\int_{\mathbb{R}} \big\vert\sum\limits_{j\in \llbracket 1,M \rrbracket} u_j\psi_j\big\vert^{2p+2}  \, \mathrm{d}x\\
&=\frac{1}{2}\sum _{j\in \llbracket 1,M \rrbracket} w_j u_j\overline{u_j} \int_{\mathbb{R}}\psi_j^2  \, \mathrm{d}x \\
&\hspace{3cm}\pm \frac{1}{2p+2} \sum_{j \in {\llbracket 1,M \rrbracket}^{2p+2}} \left(\int_{\mathbb{R}}\psi_{j_1}\cdots\psi_{j_{2p+2}}  \, \mathrm{d}x\right) u_{j_1}\cdots u_{j_{p+1}}\overline{u_{j_{p+2}}}\cdots\overline{u_{j_{2p+2}}}\\
&= \underbrace{\frac{1}{2}\sum _{j\in \llbracket 1,M \rrbracket} w_j\abs{ u_j}^2}_{Z_2(u)}\pm \underbrace{\frac{1}{2p+2}\sum_{j \in {\llbracket 1,M \rrbracket}^{2p+2}}P_j u_{j_1}\cdots u_{j_{p+1}}\overline{u_{j_{p+2}}}\cdots\overline{u_{j_{2p+2}}}}_{P(u)}.
\end{align*}
Clearly $P$ is a Hamiltonian polynomial of degree $2p+2$. Moreover, using H\"older's inequality and the fact that ${\Vert \psi_j \Vert}_{L^n} \lesssim 1$ for all $n\geq 2$, we get that\footnote{The idea is that we need to control $P_j$ in order to control the interaction between the modes via the nonlinear term.}
\begin{align*}
  \abs{ P_j}&= {\Vert \psi_{j_1}\cdots\psi_{j_{2p+2}} \Vert}_{L^1} \leq {\Vert \psi_{j_1} \Vert}_{L^{2p+2}} \cdots{\Vert \psi_{j_{2p+2}} \Vert}_{L^{2p+2}}\lesssim 1.
   \end{align*}
 At this stage, we are able to apply the Birkhoff Normal Form Theorem (recall Theorem  \fcolorbox{red}{ white}{\ref{N102}}). We obtain three positive constants $C, b,$ and $\varepsilon_0$ as well as two symplectic maps $\tau^{(0)}$ and $\tau^{(1)}$ such that the theorem holds. Note that if $C_{\rho}\varepsilon \geq 1/(C(\log M)^b),$ then we obtain
 \[\abs{\abs{u_{j_*}(t)}^2 - \abs{u_{j_*}(0)}^2} \leq \abs{u_{j_*}(t)}^2 + \abs{u_{j_*}(0)}^2 \leq {\Vert u(t) \Vert}_{h^{1/2}}^2 + {\Vert u^{(0)}  \Vert}_{h^{1/2}}^2 \leq (C_{\rho}^2+1)\varepsilon^2.\]
On the other hand, we have
$ \varepsilon^2 = \varepsilon^2 (C_{\rho}C(\log M)^b)^{2p} \frac{1}{(C_{\rho}C(\log M)^b)^{2p}} \leq  \varepsilon^2 (C_{\rho}C(\log M)^b)^{2p}  \varepsilon^{2p}$, with $ (\log M)^{2bp} \lesssim_{r,\nu} N^{2bp} \varepsilon^{-\nu}$ for $\nu>0.$ Thus, we conclude the result 
$$ \abs{\abs{u_{j_*}(t)}^2 - \abs{u_{j_*}(0)}^2} \lesssim_{\rho,r,N,\nu} \varepsilon^{2p+2-\nu}. $$ Consequently, we restrict the constants to the case $C_{\rho}\varepsilon < 1/(C(\log M)^b),$ and thus
\[ \forall t \in \mathbb{R}, \hspace{0.3cm}  {\Vert u^{\leq M}(t) \Vert}_{h^{1/2}} \leq C_{\rho}\varepsilon  < \frac{1}{C(\log M)^b} < \varepsilon_0.\]
Therefore, it \textcolor{black}{makes sense to consider} the new variable given by $v:= \tau^{(0)}\circ u^{\leq M}.$
Now, we can see that by definition of the differential and by using Lemma \fcolorbox{red}{ white}{\ref{ikea102}} we get
\begin{align*}
\partial_t v(t) &= \partial_t(\tau^{(0)}\circ u^{\leq M}(t)) = {\left( \nabla \tau^{(0)}(u^{\leq M}) , \partial_t u^{\leq M}\right)}_{\ell^2} \\
&= \mathrm{d}\tau^{(0)} (u^{\leq M})(\partial_t u^{\leq M}) = \mathrm{d}\tau^{(0)} (u^{\leq M})(-i \nabla H(u^{\leq M}) -iF^{>M}(t)) \\
&= -i (\mathrm{d}\tau^{(1)} \circ \tau^{(0)}( u^{\leq M}) )^* \hspace{0.1cm} \nabla H(u^{\leq M}) -i \mathrm{d}\tau^{(0)} (u^{\leq M} ) (F^{>M}(t)).
\end{align*}
Since the diagram commutes, $\tau^{(1)} \circ \tau^{(0)}(u^{\leq M}) = u^{\leq M}$, so we get $\tau^{(1)}(v) = u^{\leq M}.$ Then
\begin{align*}
\partial_t v(t) &= -i (\mathrm{d}\tau^{(1)} (v) )^* \hspace{0.1cm} (\nabla H)\circ \tau^{(1)}(v) -i \mathrm{d}\tau^{(0)} (u^{\leq M} ) (F^{>M}(t))\\
&= -i\left( \nabla (H \circ \tau^{(1)})(v(t)) + \mathrm{d}\tau^{(0)} (u^{\leq M} ) (F^{>M}(t))  \right).
\end{align*}
Consequently, using (\fcolorbox{red}{ white}{\ref{ikea104}}) we obtain
\begin{align} \label{ikea110}
\partial_t v(t) &= -i  \left( \nabla (Z_2+Q+R )(v(t)) + \mathrm{d}\tau^{(0)} (u^{\leq M} ) (F^{>M}(t))  \right).\end{align}
Our goal is to estimate $\partial_t \abs{v_{j_*}(t)}^2$ in order to apply the Mean Value Inequality. Since $\abs{v_{j_*}}^2$ is smooth on $h^{-1/2},$ then by composition we have that $t \mapsto \abs{v_{j_*}(t)}^2 \in \mathscr{C}^1(\mathbb{R},\mathbb{R})$ using the chain rule. So, we differentiate with respect to $t$ and use (\fcolorbox{red}{ white}{\ref{ikea110}}) to get
\begin{align*}
\partial_t \abs{v_{j_*}(t)}^2 &= {\left( i\nabla \abs{v_{j_*}}^2, i\partial_t v_{j_*} \right)}_{\ell^2} \\
&= {\left( i\nabla \abs{v_{j_*}}^2, \nabla (Z_2+Q+R )(v) \right)}_{\ell^2}  +  {\left( i\nabla \abs{v_{j_*}}^2, \mathrm{d}\tau^{(0)} (u^{\leq M} ) (F^{>M}(t))  \right)}_{\ell^2} \\
&= \{\abs{v_{j_*}}^2 , Z_2(v)+Q(v) \} +  {\left( i\nabla \abs{v_{j_*}}^2, \nabla R (v) \right)}_{\ell^2} \\
&\hspace{6.3cm} +  {\left( i\nabla \abs{v_{j_*}}^2, \mathrm{d}\tau^{(0)} (u^{\leq M} ) (F^{>M}(t))  \right)}_{\ell^2} .
\end{align*}
From Birkhoff Normal Form Theorem and by direct calculations, we can see that since $j_*=N$, we have $\{\abs{v_{j_*}}^2 , Z_2(v)+Q(v) \} =0.$ Thus, using Cauchy--Schwarz, we estimate
\begin{align*}
\abs{\partial_t \abs{v_{j_*}(t)}^2} & \leq \abs{ {\left( i \nabla \abs{v_{j_*}}^2, \nabla R (v) \right)}_{\ell^2}  } + \abs{ {\left( i\nabla \abs{v_{j_*}}^2, \mathrm{d}\tau^{(0)} (u^{\leq M} ) (F^{>M}(t))  \right)}_{\ell^2}} \\
&\leq {\Vert \nabla \abs{v_{j_*}}^2 \Vert}_{h^{1/2}}  {\Vert \nabla R(v)\Vert}_{h^{-1/2}} +  {\Vert \nabla \abs{v_{j_*}}^2 \Vert}_{h^{1/2}}  {\Vert \mathrm{d}\tau^{(0)} (u^{\leq M} ) (F^{>M}(t))  \Vert}_{h^{-1/2}}\\
&\leq 2 {\Vert v_{j_*} \Vert}_{h^{1/2}} \left(  {\Vert \nabla R(v)\Vert}_{h^{-1/2}} +   {\Vert \mathrm{d}\tau^{(0)} (u^{\leq M} ) (F^{>M}(t))  \Vert}_{h^{-1/2}}  \right).
\end{align*}
\underline{\text{Estimate of }$ {\Vert v_{j_*} \Vert}_{h^{1/2}}$:} Since $\tau^{(0)}$ is close to the identity, then 
\begin{align*} 
 {\Vert v \Vert}_{h^{1/2}} &\leq  {\Vert u \Vert}_{h^{1/2}} +  {\Vert \tau^{(0)} u -u \Vert}_{h^{1/2}} \leq 
{\Vert u \Vert}_{h^{1/2}} + \left( \frac{ {\Vert u \Vert}_{h^{1/2}}}{C_{\rho} \varepsilon} \right)^{2p} {\Vert u \Vert}_{h^{1/2}} \leq 2 {\Vert u \Vert}_{h^{1/2}} \leq 2\varepsilon_0.
\end{align*}

\underline{\text{Estimate of }$ {\Vert \nabla R(v)\Vert}_{h^{-1/2}} $:} By Theorem \fcolorbox{red}{ white}{\ref{N102}}, we have for $v \in B_{\mathbb{C}^{\llbracket 1,M \rrbracket}}(0, 2\varepsilon_0)$
\[ {\Vert \nabla R(v)\Vert}_{h^{-1/2}}  \lesssim (\log M)^b {\Vert v\Vert}_{h^{1/2}}^{2r+2p} \lesssim  (\log M)^b \varepsilon^{2r+2p}. \]

\underline{\text{Estimate of }$  {\Vert \mathrm{d}\tau^{(0)} (u^{\leq M} ) (F^{>M}(t))  \Vert}_{h^{-1/2}} $:} Again, by Theorem \fcolorbox{red}{ white}{\ref{N102}}, we obtain
\[  {\Vert \mathrm{d}\tau^{(0)} (u^{\leq M} ) (F^{>M}(t))  \Vert}_{h^{-1/2}} \leq  {\Vert \mathrm{d}\tau^{(0)} (u^{\leq M} )\Vert}_{\mathscr{L}(h^{-1/2})} {\Vert  F^{>M}(t)  \Vert}_{h^{-1/2}} \lesssim 4^r \varepsilon^{2r+2p}. \]
As a consequence, combining all the above estimations we obtain
\[ \abs{\partial_t \abs{v_{j_*}(t)}^2} \lesssim_r \varepsilon ( (\log M)^b \varepsilon^{2r+2p} + \varepsilon^{2r+2p} ) \lesssim_r  (\log M)^b \varepsilon^{2r+2p+1}.\]
Now, we apply the Mean Value Inequality on $[0,t]$:
\[ \abs{t} <\varepsilon^{-2r+1} \implies \abs{ \abs{v_{j_*}}^2 - \abs{v_{j_*}(0)}^2 } \leq  \abs{t} \abs{\partial_t \abs{v_{j_*}(t)}^2} \lesssim_r \varepsilon^{-2r+1}  (\log M)^b \varepsilon^{2r+2p+1} \lesssim_r  (\log M)^b \varepsilon^{2p+2}. \]
In order to conclude, we need to obtain a similar result for $\abs{u_{j_*}(t)}^2$. Notice that 
\begin{align}\label{ikea115}
\abs{ \abs{u_{j_*}(t)}^2  - \abs{u_{j_*}(0)}^2 } & \leq  \abs{ \abs{u_{j_*}(t)}^2  - \abs{v_{j_*}(t)}^2 } + \abs{ \abs{v_{j_*}(t)}^2  - \abs{v_{j_*}(0)}^2 } \\
\nonumber &\hspace{6.1cm}+ \abs{ \abs{v_{j_*}(0)}^2  - \abs{u_{j_*}(0)}^2 } .
\end{align}
In addition to this, we know that for all $t \in \mathbb{R}$ we have 
\begin{align*} 
\abs{ \abs{u_{j_*}(t)}^2  - \abs{v_{j_*}(t)}^2 } &\leq {\Vert u^{\leq M}(t) - v(t)\Vert}_{\ell^2}(  {\Vert v(t)\Vert}_{\ell^2} +  {\Vert u^{\leq M}(t)\Vert }_{\ell^2}) \\ &\hspace{3cm}\leq  {\Vert u^{\leq M}(t) - v(t) \Vert}_{h^{1/2}}(  {\Vert v(t)\Vert}_{h^{1/2}} +  {\Vert u^{\leq M}(t) \Vert}_{h^{1/2}})
\end{align*}
with
 \[ {\Vert u^{\leq M}(t) - v(t) \Vert}_{h^{1/2}} \leq  \left( \frac{ {\Vert u^{\leq M} \Vert}_{h^{1/2}}}{\varepsilon_0} \right)^{2p} {\Vert u^{\leq M} \Vert}_{h^{1/2}}  \lesssim   (\log M)^{2bp} \varepsilon^{2p+1} .  \]
 Finally, replacing in (\fcolorbox{red}{ white}{\ref{ikea115}}) and using that $ (\log M)^{2bp} \lesssim_{r,\nu} N^{2bp} \varepsilon^{-\nu}$ for $\nu>0,$ we deduce
 \[\abs{ \abs{u_{j_*}(t)}^2  - \abs{u_{j_*}(0)}^2 } \lesssim_r  (\log M)^{2bp} \varepsilon^{2p+2} \lesssim_{r,N,\nu} \varepsilon^{2p+2-\nu} .  \]

\begin{appendix}
\section{Appendix}\label{1001}
Here are few painless results.



\begin{lem}\label{lems}
For a given weight $P \in \widehat{H}^3$ with $P_k \in \mathbb{R}_+^*$, the assumptions (\fcolorbox{red}{ white}{\ref{my1}}) are satisfied.
\end{lem}

\begin{proof}
We start by showing that almost surely $V \in \widehat{H}^1 \cap \mathcal{C}^2.$ Indeed, 
\begin{align*}
{\Vert V \Vert}_{\widehat{H}^3}^2 &\simeq \sum_{j \geq 1} \langle j \rangle ^3 \left( V, h_j(\cdot \sqrt{2})\right)_{L^2}^2 =  \sum_{j \geq 1} \langle j \rangle ^3 {\left( \sum_{k\geq 1} g_k h_k(\cdot \sqrt{2}) P_k, h_j(\cdot \sqrt{2})\right)}_{L^2}^2 = \sum_{k\geq 1 } \langle k \rangle ^3 P_k^2g_k^2. 
\end{align*}
Setting $X= \sum\limits_{k\geq 1 } \langle k \rangle ^3 P_k^2g_k^2$ and recalling that $g_k \sim \mathcal{N}(0,1),$ we notice that
\begin{align*}
\mathbb{E}[X ] &= \mathbb{E}\big[\sum_{k\geq 1 } \langle k \rangle ^3 P_k^2g_k^2\big] = \sum_{k\geq 1 } \mathbb{E}[\langle k \rangle ^3 P_k^2g_k^2] = \sum_{k\geq 1 } \langle k \rangle ^3 P_k^2 \underbrace{\mathbb{E}[(g_k - \mathbb{E}[g_k])^2]}_{\text{ Var}( g_k)}  = {\Vert P \Vert}_{\widehat{H}^3}^2 .
\end{align*}
Since $P \in \widehat{H}^3$, we deduce that almost surely, $X$ is finite and $V$ belongs to $\widehat{H}^3 \subset \widehat{H}^1.$ Furthermore, using Sobolev embeddings we have $\widehat{H}^3 \subset H^3 \subset \mathcal{C}^2$ and thus $V \in \mathcal{C}^2.$ We turn next to proving the second assumption. For $P \in \widehat{H}^3,$ we denote by $\mathcal{K} >0 $ the sum of the convergent series $\sum\limits_{k\geq 1} \langle k \rangle (\log(k+1))^2 P_k^2$. Then for a fixed $\lambda >0,$ we have
\begin{align*}
\mathbb{P}(\, {\Vert V \Vert}_{\widehat{H}^1} < \lambda ) &\geq \mathbb{P}(\, \forall k\geq 1, \,\, \abs{g_k} < \mathcal{K}^{-1/2} \lambda \log(k+1) ) \\
&= \prod_{k\geq 1} \mathbb{P}(\, \abs{g_k} < \mathcal{K}^{-1/2} \lambda \log(k+1) )\\
  &= \prod_{k\geq 1}[1- \mathbb{P}(\, \abs{g_k} \geq \mathcal{K}^{-1/2} \lambda \log(k+1) )]\\
&= \prod_{k\geq 1}\bigg[1-  \frac{2}{\sqrt{2\pi}} \int_{\mathcal{K}^{-1/2} \lambda \log (k+1)}^{+\infty} e^{-g_k^2/2} \, \mathrm{d}g_k \bigg]\\
&=  \prod_{k\geq 1}\bigg[1-  \frac{2}{\sqrt{2\pi}} \int_{0}^{+\infty} e^{-\left(g_k+\mathcal{K}^{-1/2}\lambda \log (k+1)\right)^2/2} \, \mathrm{d}g_k \bigg]\\
&\geq  \prod_{k\geq 1}\bigg[1-  \frac{2}{\sqrt{2\pi}} e^{-\lambda^2 (\log (k+1))^2/2\mathcal{K}} \int_{0}^{+\infty} e^{-g_k^2/2} \, \mathrm{d}g_k \bigg]\\
&\geq \prod_{k\geq 1}\bigg[1- e^{-\lambda ^2 (\log (k+1))^2/2\mathcal{K}} \bigg].
\end{align*}
\textcolor{black}{Notice that since $\log(k^2+1) \lesssim_{\lambda} \lambda^2 (\log(k+1))^2/{2\mathcal{K}}$ for $k\geq 1$, $\lambda>0$ and $\mathcal{K}>0,$ then $e^{-\lambda ^2 (\log (k+1))^2/2\mathcal{K}} \lesssim_{\lambda} e^{-\log(k^2+1)} = \langle k \rangle ^{-2}$. Thus, we get $$\sum_{k\geq 1}e^{-\lambda ^2 (\log (k+1))^2/2\mathcal{K}} \lesssim_{\lambda} \sum_{k\geq 1} \langle k \rangle ^{-2}$$ which converges. }Finally, as $0< e^{-\lambda ^2 (\log (k+1))^2/2\mathcal{K}}<1$, we directly conclude that 
\[\prod_{k\geq 1}\bigg[1-  \frac{2}{\sqrt{2\pi}} e^{-\lambda ^2 (\log (k+1))^2/2\mathcal{K}} \bigg] >0. \qedhere\]
 \end{proof}

\begin{lem}\label{ike1}
There exists $K >1$ such that the following estimate holds
\[\langle j_1 +\cdots + j_{r_*}-\ell_{1}- \cdots - \ell_{r_*} \rangle \leq K (r_*+1)+\abs{\Omega_{j,\ell}(V)}.\]
\end{lem}

\begin{proof}
From Lemma \fcolorbox{red}{ white}{\ref{N7}}, we deduce that for all $j \geq 1$, there exists $C>0$ such that \[\abs{ w_j - j} \leq Cj^{-1/12} \leq C.\]
Now, consider the decomposition 
\begin{align*} 
&\langle j_1 +\cdots + j_{r_*}-\ell_{1}- \cdots - \ell_{r_*} \rangle \\
&=( w_{j_1} + \cdots +w_{j_{r_*}} -  w_{\ell_1} - \cdots - w_{\ell_{r_*}}) \\
&\hspace{2cm}+\big[\langle j_1 +\cdots + j_{r_*}-\ell_{1}- \cdots - \ell_{r_*} \rangle - ( j_1 +\cdots + j_{r_*}-\ell_{1}- \cdots - \ell_{r_*} )\big] \\
&\hspace{2cm}+ \big[(j_1 - w_{j_1}) + \cdots +  (j_{r_*} - w_{j_{r_*}}) -  (\ell_{1} - w_{\ell_1}) - \cdots -  (\ell_{r_*} - w_{\ell_{r_*}}) \big].
\end{align*}
Using the fact that for all $y \geq 0, $ we have $\abs{\langle y \rangle - y } \leq 1$, we directly establish that
\begin{align*} 
\langle j_1 +\cdots + j_{r_*}-\ell_{1}- \cdots - \ell_{r_*} \rangle  &\leq 1 +\underbrace{\sum_{n=1}^{r_*}\left(\abs{w_{j_n}-j_n} +\abs{w_{\ell_n} -\ell_n}\right)}_{\leq cr_*} + \abs{\Omega_{j,\ell}(V)}\\
&\leq \max(1,c)(r_*+1) + \abs{\Omega_{j,\ell}(V)}.
\end{align*}
\end{proof}

\begin{lem}\label{ikea102}
If $(\mathrm{d}\tau^{(1)}\circ \tau^{(0)})^*$ denotes the adjoint of $\mathrm{d}\tau^{(1)}\circ \tau^{(0)},$ then we have 
\[\mathrm{d}\tau^{(0)}i = i (\mathrm{d}\tau^{(1)}\circ \tau^{(0)})^* .\]
\end{lem}

\begin{proof}
Let $u,v$ and $w \in \mathbb{C}^{\llbracket 1,M \rrbracket}.$ Since $\tau^{(1)}$ is symplectic (recall Definition \fcolorbox{red}{ white}{\ref{defff4}}), we have 
\begin{align}
    \nonumber {\left( (\mathrm{d}\tau^{(1)}(u))^* i (\mathrm{d}\tau^{(1)})(u)(v),w\right)}_{\ell^2}&=  {\left( i(\mathrm{d}\tau^{(1)})(u)(v), (\mathrm{d}\tau^{(1)})(u)(w) \right)}_{\ell^2}  = {\left( iv,w \right)}_{\ell^2}.
      \end{align}
     This implies that for all $u \in B_{\mathbb{C}^{\llbracket 1,M \rrbracket}}(0,\varepsilon_0)$, we get $ (\mathrm{d}\tau^{(1)}(u))^* i (\mathrm{d}\tau^{(1)})(u)= i.$ 
In particular for $u= \tau^{(0)}.$ Now, since the diagram in Theorem  \fcolorbox{red}{ white}{\ref{N102}} commutes, we obtain 
 \begin{align}
 \label{eq3901}((\mathrm{d}\tau^{(1)})\circ \tau^{(0)})\mathrm{d}\tau^{(0)} = \mathrm{d}(\tau^{(1)}\circ \tau^{(0)})= \mathrm{d}(\mathrm{id}_{\mathbb{C}^{\llbracket 1,M \rrbracket}}) = \mathrm{id}_{\mathbb{C}^{\llbracket 1,M \rrbracket}} .
 \end{align}
 Finally, multiplying $ (\mathrm{d}\tau^{(1)}(u))^* i (\mathrm{d}\tau^{(1)})(u)= i$ by $\mathrm{d}\tau^{(0)}$ and using (\fcolorbox{red}{ white}{\ref{eq3901}}), we deduce
 \begin{align*}
   i( \mathrm{d}\tau^{(0)})&= (\mathrm{d}\tau^{(1)}\circ \tau^{(0)})^* i ((\mathrm{d}\tau^{(1)})\circ \tau^{(0)})\mathrm{d}\tau^{(0)} = (\mathrm{d}\tau^{(1)}\circ \tau^{(0)})^* i (\mathrm{id}_{\mathbb{C}^{\llbracket 1,M \rrbracket}}) = (\mathrm{d}\tau^{(1)}\circ \tau^{(0)})^* i.
 \end{align*}
\end{proof}

\end{appendix}
\addcontentsline{toc}{section}{Acknowledgement}
  \textbf{Acknowledgement.} I would like to thank my advisors Benoit Grébert and Joackim Bernier for their valuable comments and support throughout the process. Also, I thank Gabriel Rivière for the generous ideas and the enlightening discussions. Last but not least I am thankful for the benefits from Centre Henri Lebesgue, programme ANR-11-LABX-0020-0.

\end{document}